\documentclass[12pt]{article}

\usepackage[margin=1in]{geometry}  
\usepackage{graphicx}              
\usepackage{amsmath}               
\usepackage{amsfonts}              
\usepackage{amsthm}                
\usepackage{amssymb}

\usepackage[all, knot]{xy} 
\xyoption{arc} 
\xyoption{rotate}
\SelectTips{cm}{12}

\usepackage{chemarrow}
\usepackage{xcolor}

\usepackage{hyperref}

\let\originalparagraph\paragraph
\renewcommand{\paragraph}[2][.]{\originalparagraph{#2#1}}

\newtheorem{thm}{Theorem}[section]
\newtheorem{lem}[thm]{Lemma}
\newtheorem{prop}[thm]{Proposition}
\newtheorem{cor}[thm]{Corollary}

\theoremstyle{definition}
\newtheorem{example} [thm]{Example}
\newtheorem{rem}[thm]{Remark}
\newtheorem{definition}[thm]{Definition}

\DeclareMathOperator{\id}{id}
\DeclareMathOperator{\Id}{Id}
\DeclareMathOperator{\Eq}{Eq}
\DeclareMathOperator{\rank}{rk}
\DeclareMathOperator{\img}{im}
\DeclareMathOperator{\Hom}{Hom}
\DeclareMathOperator{\tr}{tr}

\newcommand{\R}{\mathbb{R}}      
\newcommand{\Z}{\mathbb{Z}}      
\newcommand{\N}{\mathbb{N}}      

\newdir{star}{{\star}}
\newdir{dia}{{\diamond}}

\newcommand{\Ext}{\raisebox{0.03cm}{\mbox{\footnotesize$\textstyle{\bigwedge}$}}}

\usepackage{array}
\makeatletter
\newcommand{\thickhline}{%
    \noalign {\ifnum 0=`}\fi \hrule height 1pt
    \futurelet \reserved@a \@xhline
}
\newcolumntype{"}{@{\hskip\tabcolsep\vrule width 1pt\hskip\tabcolsep}}
\makeatother

\newcolumntype{x}[1]{>{\centering\arraybackslash\hspace{0pt}}p{#1}}

\usepackage{HnOdd_diags}

\begin{document}

\nocite{*}

\title{Odd Khovanov's arc algebra}

\author{Gr\'egoire Naisse 
and Pedro Vaz\\ 
Universit\'e catholique de Louvain \\
Louvain-la-Neuve, Belgium}

\maketitle

\begin{abstract}
  We construct an odd version of Khovanov's arc algebra $H^n$.
  Extending the center to elements that anticommute, we get a subalgebra that is isomorphic to the oddification of the
  cohomology of the $(n,n)$-Springer variety.
  We also prove that the odd arc algebra
can be twisted into an associative algebra. 
\end{abstract}

\tableofcontents

\section{Introduction}

Arc algebras were first introduced by Khovanov in~\cite{khovanov02} to extend
his categorification of
Jones' link invariant~\cite{khovanov00} to tangles. 
One of his main ingredients is a certain Frobenius 
algebra of rank 2, which coincides with the cohomology ring of complex projective space. 
In a follow-up paper~\cite{khovanov04}, Khovanov showed that the arc algebra $H^n$ from~\cite{khovanov02} 
is closely related to the geometry of Springer varieties. Indeed, 
he proved that the center of $H^n$ is isomorphic to the cohomology ring of the $(n,n)$-Springer variety. 
Later on, Chen and Khovanov defined in \cite{chen-khovanov} subquotients of $H^n$ with the aim of giving an 
explicit categorification of the action of tangles on tensor powers of the fundamental representation of 
quantum $\mathfrak{sl}(2)$. To do  so, they categorified the $n$-fold tensor power of the fundamental representation of $U_q(\mathfrak{sl}(2))$ 
together with its weight space decomposition.  
Additionally, Khovanov's arc algebra was further studied in the 
sequence of papers~\cite{brundan-stroppel2,brundan-stroppel1,brundan-stroppel3,brundan-stroppel5,brundan-stroppel4}, 
where most of its interesting representation-theoretic properties were revealed.  

\medskip

Khovanov's arc algebra was later generalized in several directions by several authors.  
It was first generalized to $\mathfrak{sl}_3$-web algebras in~\cite{mackaay-tubbenhauer-pan} and then to 
$\mathfrak{sl}_n$-web algebras in~\cite{mackaay1} and further studied in~\cite{tubbenhauer1,tubbenhauer2}.
In~\cite{sartori}, a version of the arc algebra associated with $\mathfrak{gl}(1|1)$ was constructed, 
motivated by a representation-theoretic 
categorification of the Alexander polynomial.  
In~\cite{ehrig-stroppel1,ehrig-stroppel2}, a Khovanov algebra of type $D$ was introduced, 
in connection with orthosymplectic Lie algebras. 
More recently, a variant of Khovanov's arc algebra based on 
Blanchet's version of Khovanov homology~\cite{blanchet} was constructed in~\cite{ehrig-stroppel-tubbenhauer1}. This was extended in~\cite{ehrig-stroppel-tubbenhauer2} to 
$\mathfrak{gl}_2$-arc and web algebras associated with the variants of Khovanov homology from~\cite{caprau} 
and~\cite{clark-morrison-walker}. 
One of the main properties of the arc/web algebras above is that, except 
the ones from~\cite{ehrig-stroppel1,ehrig-stroppel2} for which it is not known, they all admit topological constructions using 
cobordisms or foams.

\medskip

In~\cite{ors}, Ozsvath, Rasmussen and Szabo used an exterior version of Khovanov's original Frobenius algebra to give an odd version of Khovanov homology.
Odd Khovanov homology agrees with the even (usual) Khovanov homology from~\cite{khovanov00} modulo 2, but they differ over fields 
of characteristic other than 2. 
Moreover, both categorify the Jones polynomial (see for example~\cite{shumakovitch} for further properties).  
Odd Khovanov homology was given a (Bar-Natan style~\cite{barnatan05}) topological set-up  by Putyra in~\cite{putyra14}. 
He introduced the so-called chronological cobordisms, which are cobordisms together with some extra structure related to  
a height function.

\medskip

In this paper we use the set-up from~\cite{ors,putyra14} 
and construct an odd version of Khovanov arc algebra from~\cite{khovanov02}.

\subsection{Sketch of the construction and main results}

The first step in the construction of an odd version of Khovanov's arc algebra is to replace the 
TQFT obtained by the Frobenius algebra from~\cite{khovanov00} by the chronological TQFT from~\cite{ors,putyra14}.  
As explained in~\cite{putyra14}, in order to get a well-defined category of cobordisms one has to choose 
an orientation for each of the local Morse moves. It was proved in~\cite{putyra14} (and in~\cite{ors} in an algebraic set-up) 
that any consistent choice of orientations gives the same link homology.   
This is no longer the case if one tries to extend odd Khovanov homology to tangles. 
In particular, a priori there is no reason for two different choices of orientations to result in isomorphic odd arc algebras. 

\smallskip

Our first result is that we get a family of odd arc algebras indexed by all possible choices. 
We denote $OH^n_C$ the odd arc algebra associated with the choice $C$ of chronological cobordisms.  
As a second result, we get that for all $C$ and all $n\geq2$, the odd arc algebra $OH^n_C$ is nonassociative.
This is done in Section~\ref{sect:hnodd}. 

\smallskip

In Section~\ref{sec:ospringer}, we prove an odd version of Khovanv's results from~\cite{khovanov04}. 
Namely, we prove that the odd center of  $OH^n_C$ is isomorphic to the odd cohomology of the $(n,n)$-Springer variety 
as given by Lauda and Russell in~\cite{laudarussell14}.  
In this paper they constructed an oddification of the cohomology of the Springer variety  
associated to any partition, by replacing polynomial rings and symmetric functions by their odd counterparts. 

\smallskip

As mentioned above, the algebra $OH^n_C$ is not associative. 
But this is not too big a problem, since it is a quasialgebra in the sense of Albuquerque and Majid.  
They defined in~\cite{octonions99} the notion of quasialgebra, which is a nonassociative graded algebra 
with an associator given by a $3$-cocycle coming from a higher structure, that is, a monoidal category. 
In Section~\ref{sec:assoc}, we introduce a grading on $OH^n$ by a groupoid and prove the quasi-associativity of $OH^n_C$, 
the associator depending only on $C$. 
The idea of looking at an odd version of Khovanov's arc algebra as a quasialgebra goes back to the attempts of 
Putyra and Shumakovitch to extend the odd Khovanov homology to tangles. 
An extended discussion over generalised quasialgebras can be found in the unpublished work of Putyra~\cite{putyrapreprint}.
We prove that the associator is a coboundary and thus admits a primitive $\tau$. 
Twisting the multiplication of $OH^n_C$ by this $\tau$ defines an associative algebra which keeps the odd flavor of $OH^n_C$. 
In addition, we prove that all choices of $C$ and of twist lead to isomorphic algebras. 

\medskip

\paragraph{Acknowledgments} We would like the thank Krzysztof Putyra for the discussions and ideas leading to the results of Section~\ref{sec:assoc}. 
We thank also Daniel Tubbenhauer for comments on a previous version of this paper.  
G.N. is a Research Fellow of the Fonds de la Recherche Scientifique - FNRS, under Grant no.~1.A310.16.
P.V. was supported by the Fonds de la Recherche Scientifique - FNRS under Grant no.~J.0135.16.

\section{Reminders}

To begin, we recall the three main constructions we will use: the Khovanov arc algebra, the TQFT from odd Khovanov homology and the oddification of the cohomology of the Springer varieties.

\subsection{Khovanov's arc algebra}
\label{sec:khH}

As the construction in this paper follows Khovanov's original setup from~\cite{khovanov02}, we give below a sketch of the construction of the arc algebra $H^n$.

\paragraph{Crossingless matchings} Let $B^n$ be the set of crossingless matchings of $2n$ points, that is, all ways one can pair $2n$ points on a horizontal line by non-crossing arcs placed below this line. For $b \in B^n$, we denote by $W(b)$ the reflection of $b$ across the horizontal line and by $W(b)a$ the gluing of $W(b)$ on the top of $a \in B^n$. It is clear that $W(b)a$ is a disjoint union of circles. 
 For example, we have in $B^2$:
\begin{align*}
a &=\ {\Cmda}\ , & b &=\  {\Cmdb}\ , \\
W(b) &=\  {\CmdWb}\ , & W(b)a &=\  {\xy (0,-1.6)*{\Cmda}; (0,1.6)*{\CmdWb};  \endxy}\ .
\end{align*}
We also write $W(d)cW(b)a$ for the concatenation of $W(d)c$ on top of $W(b)a$, which is the disjoint union of $W(d)c$ and $W(b)a$, see for example (\ref{eq:hnpicture}).

\paragraph{Contraction cobordisms} Given a diagram $W(c)bW(b)a$, we construct a cobordism
\begin{equation}
S_{cba} : W(c)bW(b)a \rightarrow W(c)a \label{eq:Scba}
\end{equation}
 by contracting the arcs of $b$ with their symmetric  counterparts in $W(b)$ by saddles:
\begin{equation*}
{\xy
(0,7.5)*{};
(10,7.5)*{}
**\crv{ (0,2.5) & (10,2.5)};
(0,-7.5)*{};
(10,-7.5)*{}
**\crv{ (0,-2.5) & (10,-2.5)};
(-2.5,7.5)*{}; (12.5,7.5)*{} **\dir{--}; 
(-2.5,-7.5)*{}; (12.5,-7.5)*{} **\dir{--}; 
\endxy} \quad  \xrightarrow[\text{by a saddle}]{\quad\text{arcs contraction}\quad} \quad
{\xy
(0,-7.5)*{};
(0,7.5)*{}
**\crv{ (0,-7.5) & (0,7.5)};
(10,-7.5)*{};
(10,7.5)*{}
**\crv{ (10,-7.5) & (10,7.5)};
(-2.5,7.5)*{}; (12.5,7.5)*{} **\dir{--}; %
(-2.5,-7.5)*{}; (12.5,-7.5)*{} **\dir{--}; %
\endxy}
\quad\overset{\text{rescaling}}{\simeq}\quad
{\xy
(-2.5,0)*{}; (12.5,0)*{} **\dir{--}; %
\endxy}
\end{equation*} 
This gives a surface with one saddle point for each arc in $b$. Therefore, $S_{cba}$ has (minimal) Euler characteristic $-n$, is embedded in $\R^2 \times [0,1]$ and is unique up to isotopy. Indeed, contracting the symmetric arcs in two different orders gives rise to homeomorphic surfaces and thus, the construction does not depend on any choice. Moreover, $S_{cba}$ can be given a canonical orientation. The picture to keep in mind is:
\begin{equation} \label{eq:hnpicture}
{\xy (0,2.5)*{\Cmtb}; 
(-24,5)*{}; (16,5)*{} **\dir{--}; 
(-20,10)*{W(c)};
  (0,8.25)*{\CmtWa}; 
(-20,0)*{bW(b)};
 (0,-2.5)*{\CmtWb};
(-24,-4.5)*{}; (16,-4.5)*{} **\dir{--}; 
(-20,-9.5)*{a};
 (0,-5.75)*{\Cmtc};  \endxy}
\quad \xrightarrow{\quad S_{cba} \quad} \quad 
{\xy (0,8.35)*{\CmtWa}; (0,-5.75)*{\Cmtc};
(-12.5,-4.45)*{}; (-12.5,5)*{} **\dir{-};
(-7.5,-4.45)*{}; (-7.5,5)*{} **\dir{-};
(-2.5,-4.45)*{}; (-2.5,5)*{} **\dir{-};
(2.5,-4.45)*{}; (2.5,5)*{} **\dir{-};
(7.5,-4.45)*{}; (7.5,5)*{} **\dir{-};
(12.5,-4.45)*{}; (12.5,5)*{} **\dir{-};
(-16,5)*{}; (16,5)*{} **\dir{--}; 
(-16,-4.5)*{}; (16,-4.5)*{} **\dir{--}; 
   \endxy}
\quad \simeq \quad
{\xy (0,3.05)*{\CmtWa};
(20,5)*{W(c)};
(-16,0)*{}; (20,0)*{} **\dir{--}; 
(20,-5)*{a};
 (0,-1.55)*{\Cmtc};
   \endxy}
\end{equation}

\paragraph{Frobenius algebra} Let $A := \Z[X]/(X^2)$ be the $\Z$-graded abelian group with grading given by $\deg 1 = -1$ and $\deg X = 1$. This group possesses the structure of a $\Z$-algebra when equipped with the polynomial multiplication. However, notice that this multiplication has degree $1$ and does not gives a graded algebra structure.
 We turn $A$ into a Frobenius algebra by defining a trace,
\begin{align*}
\tr : A \rightarrow \Z,\qquad \tr(1) = 0,\quad \tr(X) = 1.
\end{align*}
As the trace is non-degenerate, this defines a TQFT, 
$$F : 2Cob \rightarrow \Z\text{-grmod},$$
 where $\Z$-grmod is the category of $\Z$-graded  free $\Z$-modules with finite rank and $2Cob$ the category of oriented cobordisms between 1-manifolds, see \cite{kock03}.   From now on, unless stated otherwise, we will always assume that graded means $\Z$-graded.

Thus, we get $F(W(b)a) \simeq A^{\otimes |W(b)a|}$ for $|W(b)a|$ the number of circle components in $W(b)a$. Moreover, the comultiplication map is explicitly given by
\begin{align*}
\Delta : A \rightarrow A \otimes A,\qquad \Delta(1) = X \otimes 1 + X \otimes 1,\quad \Delta(X) = X \otimes X.
\end{align*}

Applying this TQFT on the cobordism (\ref{eq:Scba}), we get a morphism,
\begin{equation}F(W(c)b) \otimes_\Z F(W(b)a) \simeq F(W(c)bW(b)a) \xrightarrow{F(S_{cba})} F(W(c)a). \label{eq:cobind}\end{equation}
This morphism has degree $n$ since the multiplication and comultiplication maps in $A$ have degree $1$ and $S_{cba}$ possesses $n$ saddle points.

\paragraph{Arc algebra} Define the graded abelian groups
\begin{align*}
H^n &:= \bigoplus_{a,b \in B^n} b(H^n)a, &  b(H^n)a &:= F(W(b)a)\{n\},
\end{align*}
where the notation $\{n\}$ means that we shift the degree up by $n$.
Therefore,  as the maximal number of components in $W(b)a$ is $n$, every element $x \in F(W(b)a)$ has degree $\deg x \ge -n$ and thus, $H^n$ is a $\Z_+$-graded group.
In order to define a multiplication in $H^n$, we first let the product  $d(H^n)c \otimes_\Z b(H^n)a \rightarrow H^n$ be zero whenever $c \ne b$. Then, for the other cases, we define the multiplication such that the diagram
$$\xymatrix{
c(H^n)b \otimes_\Z b(H^n)a \ar[d]_{\simeq} \ar[r] &  c(H^n)a \\
 F(W(c)b) \otimes_\Z F(W(b)a) \{2n\} \ar[r]_-{(\ref{eq:cobind})}&  F(W(c)a)\{n\} \ar[u]_{\simeq}
}$$
commutes. The associativity of the multiplication follows from the fact that $F$ is a TQFT. Moreover, the sum $\sum_{a\in B^n} 1_a$, with $1_a$ the unit in $a(H^n)a \simeq A^{\otimes n}\{n\}$, is a unit for $H^n$. All of this sums up to: 

\begin{prop}\emph{(Khovanov, \cite[Proposition 1]{khovanov02})}
The structures above make $H^n$ into a $\Z_+$-graded associative unital $\Z$-algebra.
\end{prop}

\subsection{Odd Khovanov homoloy}\label{sec:oddkh}

Ozsvath, Rasmussen and Szabo constructed in \cite{ors} an odd version of Khovanov homology using some ``projective TQFT" replacing $F$ (projective meaning here that it is well-defined only up to sign). Putyra extended in \cite{putyra14} the work of Bar-Natan for Khovanov homology~\cite{barnatan05} by giving a topological framework for the odd homology: the chronological cobordisms. In addition, Putyra's work allows the construction of the odd Khovanov homology using a well-defined functor. In this subsection, we mainly follow the exposition in \cite{putyra14}.

\paragraph{Chronological cobordims} Recall that a chronological $2$-cobordism is a $2$-cobordism equipped with a chronology, that is, a Morse function with one critical point at each critical level. Moreover, at each critical point, we choose an orientation of the space of unstable directions in the gradient flow induced by the chronology. We write that choice by an arrow.
These chronological $2$-cobordisms, taken up to isotopy which preserves the orientations and the chronology, form a category with composition given by gluing. We denote it by $2ChCob$. Every chronological $2$-cobordism can be built from the six elementary chronological $2$-cobordisms:
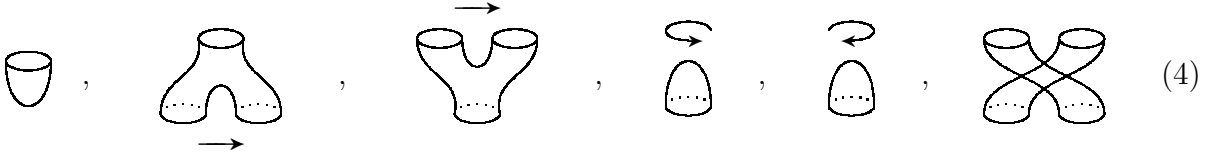
\begin{align}\label{eq:elemchcob}
{\cobcup}
 &\quad ,&
{\cobmergepos}
 &\quad ,&
{\cobsplitpos}
 &\quad ,&
{\cobcappos}
 &\quad ,&
{\cobcapneg}
&\quad ,&
{\cobtwist}
\end{align}
which are called respectively a birth, a merge, a split, a positive death, a negative death and a twist. As we are only interested in chronological $2$-cobordisms, we will forget the prefix $2$-.


\paragraph{The odd functor} We describe the functor $OF : 2ChCob \rightarrow \Z$-grmod from \cite{ors}. Morally, objects of $2ChCob$ are disjoint unions of circles. For $S$ such a union we denote by $V(S)$ the free abelian group generated by the components of $S$ with a grading such that each generator has degree $2$. We define 
$$OF(S) := \Ext^* V(S)\{-|S|\},$$
with $\Ext^* V(S)$ being the exterior algebra generated by the elemens of $V(S)$ and $|S|$ the number of components. 

\medskip 

We now define the functor on each of the elementary cobordisms (\ref{eq:elemchcob}). Let $S_1$ and $S_2$ be objects of $2ChCob$ with $S_2$ containing one circle more than $S_1$. For a birth of a circle from $S_1$ to $S_2$, there is a canonical inclusion $V(S_1) \subset V(S_2)$ (the new generator being the circle cupped by the birth cobordism). This induces a morphism
$$OF\left(\cobcup\right) :  \Ext^* V(S_1) \xrightarrow{\subset}  \Ext^* V(S_2), \quad v \mapsto 1 \wedge v.$$

Consider a merge of two circles $a_1,a_2$ in $S_2$ to a single one in $S_1$ with an arrow $a_1 \chemarrow a_2$. The arrow represent one of the two possible choices of orientation of the merge, the other being denoted $a_2 \chemarrow a_1$. There is an isomorphism of groups $V(S_1) \simeq V(S_2)/\{a_1-a_2\}$ and thus the canonical projection $V(S_2) \rightarrow V(S_2)/ \{a_1 - a_2\}$ induces a morphism
$$OF\left(\xy
(0,0)*{\cobmergepos};
(-5,-8)*{a_1};
(5.5,-8)*{a_2};
\endxy\right) :  \Ext^* V(S_2)  \rightarrow  \Ext^*  \left(V(S_2)/ \{a_1 - a_2\}\right)  \simeq \Ext^* V(S_1). $$
It is not hard to see that the choice of orientation does not change the result in this case and we get
$$OF\left(\xy
(0,0)*{\cobmergepos};
(-5,-8)*{a_1};
(5.5,-8)*{a_2};
\endxy\right) =
OF\left(\xy
(0,0)*{\cobmergeneg};
(-5,-8)*{a_1};
(5.5,-8)*{a_2};
\endxy\right).$$

Now say we have a split sending $a \in S_1$ to $b_1 \chemarrow b_2$ in $S_2$. Again, there is a natural identification $V(S_2) \simeq V(S_1)/\{b_1 - b_2\}$, but now we also use the isomorphism 
$$\Ext^*  \left(V(S_1)/\{b_1 - b_2\}\right) \simeq (b_1 - b_2) \wedge \Ext^* V(S_1)$$
to get a morphism
$$OF\left(\xy
(0,0)*{\cobsplitpos};
(-5,8)*{b_1};
(5.5,8)*{b_2};
\endxy\right) : \Ext^* V(S_2) \simeq (b_1-b_2)\wedge \Ext^* V(S_1) \xrightarrow{\subset}  \Ext^* V(S_1). $$
As a matter of fact, this morphism is easily computable by replacing the occurences of $a$ by $b_1$ (or $b_2$) and multiplying by $(b_1-b_2)$. For example, $1$ is sent to $b_1 - b_2$ and $a$ is sent to $(b_1-b_2)\wedge b_1 = b_1 \wedge b_2$. Notice also that reversing the orientation changes the sign of the morphism
$$OF\left(\xy
(0,0)*{\cobsplitpos};
(-5,8)*{b_1};
(5.5,8)*{b_2};
\endxy\right) 
=
-OF\left(\xy
(0,0)*{\cobsplitneg};
(-5,8)*{b_1};
(5.5,8)*{b_2};
\endxy\right).$$

Suppose we have a positive (in other words, anticlockwise oriented) death of $a \in S_2$. We associate to it the morphism given by contraction with the dual of $a_1$
$$OF\left(\xy
(0,0)*{\cobcappos};
(0,-8)*{a_1};
\endxy\right) : \Ext^* V(S_2) \rightarrow \Ext^* V(S_1), \quad v \mapsto a_1^*(v).
$$
The negative one is given by the opposite.

Finally, the twist is given by a the permutation of the corresponding terms
$$OF\left(\xy
(0,0)*{\cobtwist};
(-5,-9)*{a_1};
(5.5,-9)*{a_2};
(-5,9)*{a_1};
(5.5,9)*{a_2};
\endxy\right) :  \Ext^* V(S_2)  \rightarrow  \Ext^*  \left(V(S_2)\right) , \quad \begin{cases}
a_1 &\mapsto a_2, \\
a_2 &\mapsto a_1, \\
a_1 \wedge a_2 &\mapsto a_2 \wedge a_1.
\end{cases} $$

\begin{rem}\label{rem:mergeorient}
Since changing the orientations of the merges does not change the result of the functor, we will ignore them in our discussion.
\end{rem}

\subsection{Odd cohomology of the Springer varieties}

First, let us recall the definition of a Springer variety.

\begin{definition}
Let $\lambda = (\lambda_1, \dots , \lambda_m)$ be a partition of $m$, $E_m$ be a complex vector space of dimension $m$ and $z_\lambda : E_m \rightarrow E_m$  be a nilpotent linear endomorphism with $|\lambda|$ nilpotent Jordan blocks of size $\lambda_1, \dots, \lambda_m$. The Springer variety for the partition $\lambda$ is 
$$\mathfrak B_{\lambda} := \{\text{complete flags in $E_m$ stabilized by $z_\lambda$}\}.$$
\end{definition}
The cohomology ring of $\mathfrak B_{\lambda}$ can be computed by quotienting the polynomial ring in $m$ variables by the ideal of partially symmetric functions (see \cite{conciniprocesi} for more details). Write $(n,n)$ for the partition $\lambda = (n,n)$ of $2n$.

\begin{thm}\emph{(Khovanov, \cite[Theorem 1.1]{khovanov04})} \label{thm:isozhnhbnn}
There is an isomorphism of graded algebras
$$Z(H^n) \simeq H(\mathfrak B_{n,n}, \Z).$$
\end{thm}

Lauda and Russell constructed in \cite{laudarussell14} an oddification of the cohomology of the Springer varieties, denoted $OH(\mathfrak B_\lambda, \Z)$. Like the usual cohomology is obtained as a quotient of the polynomials by the
partially symmetric functions, they constructed $OH(\mathfrak B_\lambda, \Z)$ as a quotient of the ring $OPol_{m}$ of odd polynomials 
\begin{align*}
OPol_{m} &:= \frac{\Z\langle x_1, \dots, x_{m}\rangle}{\langle x_ix_j + x_jx_i = 0 \text{ for all } i\ne j \rangle}, &\deg(x_i) &= 2,
\end{align*}
by some ideal.
Since we only need the case $m=2n$ and $\lambda = (n,n)$ for our discussion, we restrict to this case from now on.

\begin{definition}{(Lauda \& Russell, \cite{laudarussell14})}\label{def:oddcoh}
The odd cohomology of the $(n,n)$-Springer variety is the quotient
$$OH(\mathfrak{B}_{n,n}, \Z) := OPol_{2n}/OI_n$$
where $OI_n$ is the left ideal generated by the set of odd partially symmetric functions
\begin{align*}
OC_n := \left \{ \epsilon_r^S := \sum_{1\le i_1 < \dots < i_r \le 2n} x_{i_1}^S\dots x_{i_r}^S \bigg| 
\parbox{15em}{$k\in \{1,2,\dots,n\}$, $|S| = n+k$,\\ \centering{$ r \in \{n-k+1, n-k, \dots,n+k\}$}} 
\right\},
\end{align*}
for all $S$ ordered subset of $\{1, \dots, 2n\}$ of cardinality $n+k$ and
$$x_{i_j}^S := \begin{cases}
0, &\text{ if } i_j \notin S,\\
(-1)^{S(i_j)-1}x_{i_j}, &\text{otherwise,}
\end{cases}$$
with $S(i_j)$ the position of $i_j$ in $S$.
\end{definition}

In general, the odd cohomology of a Springer variety is only a module over the odd polynomials. However, in case $\lambda = (n,n)$, it is a graded algebra. This is due to the fact that, thanks to \cite[Lemma~3.6]{laudarussell14}, $x_i^2 \in OI_n$ for all $i$. Thus, $OI^n$ is a $2$-sided ideal. As a matter of fact, $OH(\mathfrak{B}_{n,n}, \Z)$ also possesses the structure of a superalgebra with superdegree given by dividing the degree by~$2$. Finally, by construction, the algebra $OH(\mathfrak{B}_{n,n}, \Z)$ is isomorphic to $H(\mathfrak{B}_{n,n}, \Z)$ modulo $2$.

\begin{example}
$OH(\mathfrak{B}_{2,2}, \Z)$ is given by the odd polynomials in 4 variables $x_1,x_2,x_3,x_4$ quotiented by the (not minimal) relations:
\begin{align*}
x_1 - x_2 + x_3 - x_4 &= 0, \\
-x_ix_j+x_ix_k-x_jx_k &= 0, &\forall i<j<k \in [1,4] , \\
-x_1x_2+x_1x_3-x_1x_4-x_2x_3+x_2x_4-x_3x_4 &= 0, \\
x_ix_jx_k &= 0, &\forall i<j<k \in [1,4] , \\
-x_1x_2x_3+x_1x_2x_4-x_1x_3x_4+x_2x_3x_4 &= 0, \\
x_1x_2x_3x_4 &= 0.
\end{align*}
\end{example}


\section{Odd arc algebra}
\label{sect:hnodd}

In this section, we construct an odd version of the Khovanov arc algebra $H^n$. Therefore, we will closely follow the construction from above, replacing the TQFT $F$ by the odd functor $OF$ from Section~\ref{sec:oddkh}.
First, we define for all $n\ge 0$ the following graded abelian groups,
\begin{align*}
OH^n &:= \bigoplus_{a,b \in B^n} b(OH^n)a, &  b(OH^n)a &:= OF(W(b)a)\{n\},
\end{align*}
such that $OH^n$ is $\Z_+$-graded.

The first difficulty we encounter when we try to define a multiplication as in Section~\ref{sec:khH} is that we have to choose a chronology and signs for the splits. In the odd Khovanov homology from \cite{ors}, the signs are forced by the requirement that the cube of resolutions anticommutes (all possible choices  leading to isomorphic cubes). However, in our case, there is no condition other than that the cobordisms must be embedded in $\R^2 \times [0,1]$. This means we have to consider all possible choices.

\paragraph{Contraction cobordisms}
For each $a,b,c \in B^n$, there is a canonical cobordism with minimal number of critical points (up to homeomorphism and embedded in $\R^2 \times [0,1]$) from the diagram $W(c)bW(b)a$ to $W(c)a$. This corbordism is given by contracting the arcs of $b$ with their symmetric counterparts in $W(b)$, as in the definition of $H^n$ in Section~\ref{sec:khH}. To be able to apply $OF$, we need to define a chronological cobordism and there are several ways to do so:
\begin{itemize}
\item We have
to choose a chronology, in other words, we have to choose an order in which we contract the symmetric arcs of $bW(b)$, taking
care of never contracting two arcs before the one surrounding them.
This is required to get an embedded surface.
\item We have to give an orientation for the critical points, especially for the splits (we do not need to orient the merges by Remark~\ref{rem:mergeorient}). We express the two possibilities by an arrow:
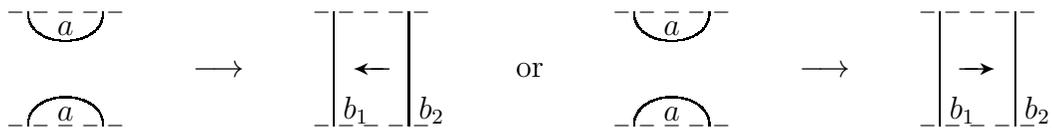
\begin{align*}
{\xy
(5,-6)*{a};
(0,7.5)*{};
(10,7.5)*{}
**\crv{ (0,2.5) & (10,2.5)};
(5,5.5)*{a};
(0,-7.5)*{};
(10,-7.5)*{}
**\crv{ (0,-2.5) & (10,-2.5)};
(-2.5,7.5)*{}; (12.5,7.5)*{} **\dir{--}; 
(-2.5,-7.5)*{}; (12.5,-7.5)*{} **\dir{--}; 
\endxy}  &\qquad \longrightarrow \qquad
{\xy
(3,-5.5)*{b_1};
(13,-5.5)*{b_2};
(0,-7.5)*{};
(0,7.5)*{}
**\crv{ (0,-7.5) & (0,7.5)};
(5,0)*{\larrowfill{14pt}};
(10,-7.5)*{};
(10,7.5)*{}
**\crv{ (10,-7.5) & (10,7.5)};
(-2.5,7.5)*{}; (12.5,7.5)*{} **\dir{--}; 
(-2.5,-7.5)*{}; (12.5,-7.5)*{} **\dir{--}; 
\endxy}
&\text{ or }&&
{\xy
(5,-6)*{a};
(0,7.5)*{};
(10,7.5)*{}
**\crv{ (0,2.5) & (10,2.5)};
(5,5.5)*{a};
(0,-7.5)*{};
(10,-7.5)*{}
**\crv{ (0,-2.5) & (10,-2.5)};
(-2.5,7.5)*{}; (12.5,7.5)*{} **\dir{--}; 
(-2.5,-7.5)*{}; (12.5,-7.5)*{} **\dir{--}; 
\endxy}  &\qquad\longrightarrow\qquad
{\xy
(3,-5.5)*{b_1};
(13,-5.5)*{b_2};
(0,-7.5)*{};
(0,7.5)*{}
**\crv{ (0,-7.5) & (0,7.5)};
(5,0)*{\rarrowfill{14pt}};
(10,-7.5)*{};
(10,7.5)*{}
**\crv{ (10,-7.5) & (10,7.5)};
(-2.5,7.5)*{}; (12.5,7.5)*{} **\dir{--}; 
(-2.5,-7.5)*{}; (12.5,-7.5)*{} **\dir{--}; 
\endxy}
\end{align*}
meaning that we split the component $a$ in two components $b_1, b_2$ with the orientation $b_2 \chemarrow b_1$ in the first case and $b_1 \chemarrow b_2$ in the second one.
\end{itemize}

\begin{rem}\label{rem:standardchoice}
There is always at least one possible choice: it suffices to go through the end points of $b$ from left to right and contracting whenever we encounter an arc which was not already contracted, then orienting the splits from left to right (i.e. putting an arrow from the component passing through the left point to the one passing through the right point). 
\end{rem}

We now assume that for each triplet $a,b,c \in B^n$ we have chosen a chronological cobordism. We write it $C_{cba}$ and denote the collection all of them by $C := \{C_{cba} | a,b,c \in B^n\}$. In addition, we write $\mathcal C^n$ for the set of all possible choices of such a set $C$. 

\paragraph{Multiplication} Like in the even case, we let the multiplication 
$$d(OH^n)c \otimes_\Z b(OH^n)a \rightarrow \{0\} \subset OH^n$$
 be zero for $c \ne b$. We define the multiplication
$c(OH^n)b \otimes_\Z b(OH^n)a \rightarrow c(OH^n)a$
using the morphism $OF(C_{cba})$. More precisely, there is a morphism
\begin{equation}
OF(W(c)b) \otimes_\Z OF(W(b)a) \rightarrow OF(W(c)bW(b)a) : (x,y) \mapsto x \wedge y \label{eq:wedge}
\end{equation}
induced by the inclusions $W(c)b \subset W(c)bW(b)a$ and $W(b)a \subset W(c)bW(b)a$. We compose it with $OF(C_{cba})$ to obtain the multiplication by making the following diagram commutes:
$$ \xymatrix{
c(OH^n)b \otimes_\Z b(OH^n)a \ar[d]_{\simeq}  \ar[rrr] &&& c(OH^n)a \\
 OF(W(c)b) \otimes_\Z OF(W(b)a) \ar[r]_-{(\ref{eq:wedge})}  & OF(W(c)bW(b)a) \ar[rr]_-{OF(C_{cba})} && \ar[u]_{\simeq} OF(W(c)a).
}$$
This map is degree preserving thanks to the minimality hypothesis on the Euler Characteristic of $C_{cba}$ which guarantees that the degree of $OF(C_{cba})$ is $n$.

\paragraph{Unit} We write $1_a$ for the unit 
 in the exterior algebra $\Ext^* V(W(a)a)$ and we easily check that the sum $\sum_{a \in B^n} 1_a$ is a unit for the multiplication defined in the paragraph above. We also write ${_b1_a}$ for the unit in the exterior algebra $\Ext^* V(W(b)a)$ (notice ${_b1_a}$ is not an idempotent). 

\begin{prop}\label{prop:nonassoc}
For $n \ge 2$ and any choice $C \in \mathcal C^n$, the multiplication defined above is not associative.
\end{prop}

\begin{proof}
First, suppose that $n =2$ and
let $C \in \mathcal C^2$ be an arbitrary choice of cobordisms. Consider $a,b \in B^2$ such that
\begin{align*}
a &=  {\Cmda}\ , & b &= {\Cmdb}\ .
\end{align*}
Take $x = b_1 \in b(OH^n)b, y = {_b1_a} \in b(OH^n)a$ and $z = {_a1_b} \in a(OH^n)b$, where $b_1$ is the element in the exterior algebra coming from the outer circle in the diagram $W(b)b$. We write $b_2$ for the element generated by the inner circle. Then we compute $x(yz)$ as follows. The cobordism $$C_{bab} : {\xy (0,8)*{}; (0,-8)*{};  (0,4)*{\yo}; (0,-4)*{\xo}\endxy}\rightarrow \bobt$$
is given by a merge followed by a split such that the element $yz$ is 
$$yz =  \alpha(b_1 - b_2), $$
where $\alpha \in \{\pm 1\}$ depends on the chosen orientation for the unique split in $C_{bab}$. Then, $x(yz)$ is given by
$$C_{bbb} : {\xy (0,8)*{}; (0,-8)*{};  (0,4)*{\bobt}; (0,-4)*{\bobt}\endxy} \rightarrow \bobt, \quad x(yz) =  -\alpha b_1 \wedge b_2$$
since $C_{bbb}$ is composed by two merges.
Now, we compute $(xy)z$ by
\begin{align*}
C_{bba} : {\xy  (0,8)*{}; (0,-8)*{};  (0,4)*{\bobt}; (0,-4)*{\yo}\endxy} \rightarrow \yo,& \quad xy = c_1, 
\intertext{with $c_1$ the element coming from the unique circle in $W(b)a$ and then}
C_{bab} : {\xy (0,8)*{}; (0,-8)*{};(0,4)*{\yo}; (0,-4)*{\xo}\endxy} \rightarrow \bobt,& \quad (xy)z = \alpha b_1 \wedge b_2.
\end{align*}
This means that for every $C \in \mathcal C^2$, $(xy)z = -x(yz)$.
To conclude the proof, we observe that this example can be extended for all $n \ge 2$ by adding the same arcs at right of $a$ and~$b$,
\begin{align*}
a_n &:= {\Cmda} {\xy(0,0)*{}; (5,0)*{}; (10,0)*{} **\crv{ (5,-2.5) & (10,-2.5)}; 
(2.5,0)*{}; (12.5,0)*{} **\dir{--}; \endxy} \quad\dots\quad {\xy (0,0)*{}; (5,0)*{} **\crv{ (0,-2.5) & (5,-2.5)};(-2.5,0)*{}; (7.5,0)*{} **\dir{--}; \endxy}\ , & b_n &:= {\Cmdb}  {\xy(0,0)*{}; (5,0)*{}; (10,0)*{} **\crv{ (5,-2.5) & (10,-2.5)};(2.5,0)*{}; (12.5,0)*{} **\dir{--}; \endxy} \quad\dots\quad {\xy (0,0)*{}; (5,0)*{} **\crv{ (0,-2.5) & (5,-2.5)};(-2.5,0)*{}; (7.5,0)*{} **\dir{--}; \endxy}\ ,
\end{align*}
such that the exact same computation can be done for all $n \ge 2$.
\end{proof}

\begin{definition}
We denote by $OH^n_C$ the  $\Z_+$-graded nonassociative, unital $\Z$-algebra given by $OH^n$ with the multiplication obtained from a $C \in \mathcal{C}^n$.
\end{definition}

\begin{rem}
We sometimes write $b(OH^n_C)a$. By this, we mean that we take the elements of the group $b(OH^n)a$, but viewed as elements in $OH^n_C$.
\end{rem}

As for each $C$ we get a family of algebras $OH^n_C$, it is legitimate to ask if on can classify them.
We give some partial answer to this question in Section~\ref{sec:assoc}.

\begin{rem}
From now on, unless otherwise specified, all assertions are valid for all  $n \in \N$ and $C \in \mathcal C^n$ and we don't specify them.
\end{rem}

To ensure that $OH^n_C$ is an odd version of $H^n$, we have to check that the two algebras agree modulo $2$.

\begin{prop}
There is an isomorphism of graded algebras
$$H^n \otimes_\Z \Z/2\Z \simeq OH^n \otimes_\Z \Z/2\Z.$$
\end{prop}

\begin{proof}
The result follows directly from the construction since the functor $F$ used to define $H^n$ agrees up to sign with $OF$.
\end{proof}

It is interesting (and it will be useful) to notice that $OH^n_C$ contains a collection of exterior algebras as subalgebras.

\begin{prop} \label{prop:extalg}
There is an inclusion of graded algebras
$$\bigoplus_{a \in B^n} \Ext^* \Z^n \simeq \bigoplus_{a \in B^n} a(OH^n_C)a \subset OH^n_C.$$
\end{prop}

\begin{proof}
It is enough to notice that $OF(W(a)a) \simeq \Ext^* \Z^n$  as $W(a)a$ is a collection of $n$ circles, and to remark that the multiplication
$$a(OH^n_C)a \otimes_\Z a(OH^n_C)a \rightarrow a(OH^n_C)a$$
is the usual product in the exterior algebra. Indeed, the cobordism $C_{aaa}$ consists of $n$ merges and no split and thus gives the exterior product.
\end{proof}

\paragraph{Diagrammatic notation}
To simplify the notation, we propose a way to write the generators of $OH^n$ using diagrams. First, notice that there is an order on the components of $W(b)a$, for $a,b \in B^n$,  given by reading the diagram from left to right. More precisely, for $a_1, a_2 \in W(b)a$, we say that $a_1 < a_2$  whenever $a_1$ passes through an end point of $a$ (or equivalently $b$) which is at the left of all end points contained in $a_2$. 

An element $x_1 \wedge \dots \wedge x_k$ in $b(OH^n)a$ is written as the diagram $W(b)a$ where we draw the components with a plain line for each $x_i$ and a dashed one for the others. Moreover, we require that $x_1 \wedge \dots \wedge x_k$ is in the order induced by reading the diagram from left to right, otherwise we add a sign to recover this order.   Thus, we get for example in $OH^2$:
\begin{align*}
\boid &= b_1 \wedge 1, & -\bobt &= b_2 \wedge b_1,\\ \\
\idab &= {_a1_b}, & \yo &= c_1,
\end{align*}
with $a_i, b_i$ and $c_i$ following the conventions from the proof of Proposition~\ref{prop:nonassoc}.

\subsection{An example : $OH^2_C$} \label{ex:oh2}
We construct explicit multiplication tables for $OH^2_C$, with $C$ the choice from Remark \ref{rem:standardchoice}. 
  Since the multiplication maps for $*_2(OH^2)a \otimes b(OH^2)*_1 \rightarrow 0$ and $*_2(OH^2)b \otimes a(OH^2)*_1 \rightarrow 0$ are zero for all $*_1, *_2 \in \{a,b\}$, we give the tables only for $*_2(OH^2)a \otimes a(OH^2)*_1 \rightarrow *_2(OH^2)*_1$ and $*_2(OH^2)b \otimes b(OH^2)*_1 \rightarrow *_2(OH^2)*_1$. Moreover, these tables are written with the convention:
$$\begin{tabular}{c|c|}
\cline{2-2}
 & y \\\hline
\multicolumn{1}{ |c |} x &  xy \\\hline
\end{tabular}$$
By direct computation we get
\begin{center}
\begin{tabular}{|c"c|c|c|c|c|c|}
  \hline
 $OH^2_C$ & $\idaa$ & $\aoid$ & $\idat$ & $\aoat$ & $\idab$ & $\xo$ \\
  \thickhline
$\idaa$ & $\idaa$ & $\aoid$ & $\idat$ & $\aoat$ &  $\idab$ & $\xo$  \\
  \hline
$\aoid$ & $\aoid$ & $0$ & $\aoat$ & $0$  & $\xo$ & $0$  \\
  \hline
$\idat$ & $\idat$ & $-\aoat$ & $0$ & $0$ & $\xo$ & $0$ \\
  \hline
$\aoat$ & $\aoat$ & 0 & 0 & 0 & 0 & 0  \\
  \hline
${\idba}$ & ${\idba}$ & $\yo$ & $\yo$ & 0 & $\idbt - \boid$ & $-\bobt$  \\
  \hline
$\yo$ & $\yo$ & 0 & 0 & 0 & $- \bobt$ & 0 \\
  \hline
\end{tabular}
\end{center}
and
\begin{center}
\begin{tabular}{|c"c|c|c|c|c|c|}
  \hline
 $OH^2_C$ & $\idbb$ & $\boid$ & $\idbt$ & $\bobt$ & $\idba$ & $\yo$ \\
  \thickhline
$\idbb$ & $\idbb$ & $\boid$ & $\idbt$ & $\bobt$ &  $\idba$ & $\yo$  \\
  \hline
$\boid$ & $\boid$ & $0$ & $\bobt$ & $0$  & $\yo$ & $0$  \\
  \hline
$\idbt$ & $\idbt$ & $-\bobt$ & $0$ & $0$ & $\yo$ & $0$ \\
  \hline
$\bobt$ & $\bobt$ & 0 & 0 & 0 & 0 & 0  \\
  \hline
$\idab$ & $\idab$ & $\xo$ & $\xo$ & 0 & $\aoid - \idat$ & $\aoat$  \\
  \hline
$\xo$ & $\xo$ & 0 & 0 & 0 & $\aoat$ & 0 \\
  \hline
\end{tabular}
\end{center}


\subsection{The odd center of $OH^n_C$}

When talking about exterior algebras (or in general superalgebras), it is common to consider the supercenter which is an extension of the center to the elements that anticommute. In the same spirit, we define the odd center for $OH^n_C$. 

\begin{definition}
We define the \emph{parity} of an homogeneous element $z \in a(OH^n)b$ by
$$p(z) := \frac{\deg(z) - \deg({_a1_b})}{2} = \frac{\deg(z) - n + |W(b)a|}{2} \mod 2,$$
with $|W(b)a|$ the number of circle components in $W(b)a$. 
\end{definition}

One can easily see that this number counts the factors of $z = a_1 \wedge ... \wedge a_m$, i.e.
$$p(a_1 \wedge \dots \wedge a_m) = m \mod 2.$$

\begin{definition}We call \emph{odd center} of $OH^n_C$ the subset
$$OZ(OH^n_C) := \left\{ z \in OH^n_C | zx = (-1)^{p(x)p(z)}xz, \forall x \in OH^n_C \right\}.$$
\end{definition}

\begin{rem}
The parity does not give a grading on $OH^n_C$, since there are elements $x,y$ in $OH^n_C$ such that
$$p(xy) \ne p(x) + p(y) \mod 2.$$
This means that $OH^n_C$ is not a superalgebra with respect to $p$. As a matter of fact, the parity descends to a degree on the antisymmetric subalgebra $\bigoplus_{a \in B^n} a(OH^n_C)a \subset OH^n_C$, in which the odd center lives. 
\end{rem}

\begin{prop}\label{prop:inccenter}
There are inclusions $Z(OH^n_C) \subset OZ(OH^n_C) \subset \bigoplus_{a\in B^n} a(OH^n_C)a$.
\end{prop}

\begin{proof}
The second inclusion is immediate since, for every $z\in OZ(OH^n_C)$, one can decompose $z = \sum_{a,b \in B^n} {_bz_a}$ with ${_bz_a} \in b(OH^n_C)a$ and get ${_bz_a} = 1_bz1_a = z(1_b1_a) = 0$, unless~$b = a$. The first inclusion is obtained by first noticing that $Z(OH^n_C) \subset \bigoplus_{a\in B^n} a(OH^n_C)a$ by the same argument as before, and then observing that every central element has even parity.
\end{proof}

Moreover, one can check that the odd center is an \emph{associative} superalgebra with superdegree given by the parity, and is characterized by the following property.

\begin{prop}\label{prop:caractoddcenter}
An element $z = \sum_{a\in B^n} z_a$ is in $OZ(OH^n_C)$ if and only if $z_b\cdot{_b1_a} = {_b1_a}\cdot z_a$ for all $a,b \in B^n$.
\end{prop}

\begin{proof}
An element $z = \sum_{a\in B^n} z_a$ commutes with $x  = \sum_{a,b \in B^n} {_bx_a}$ if and only if $z$ commutes with every~${_bx_a}$. Moreover, $z\cdot{_bx_a} = (z_b\cdot{_b1_a})\wedge{_bx_a} = (-1)^{p(z)p(x)}  {_bx_a}\wedge (z_b\cdot{_b1_a})$ and we have ${_bx_a}\wedge (z_b\cdot{_b1_a}) = {_bx_a}\wedge ({_b1_a}\cdot z_a) = {_bx_a}.z$ if and only if $z_b\cdot {_b1_a} = {_b1_a}\cdot z_a$.
\end{proof}

The following result allows us to write $OZ(OH^n)$ with no ambiguity.

\begin{prop}
For all $C, C' \in \mathcal{C}^n$, there is an isomorphism of graded (super)algebras
$$OZ(OH^n_C) \simeq OZ(OH^n_{C'}).$$
\end{prop}

\begin{proof}
The condition $z_b\cdot {_b1_a} = {_b1_a}\cdot z_a$ from Proposition \ref{prop:caractoddcenter} depends only on the multiplication maps
\begin{align*}
b(OH^n)b \otimes_\Z b(OH^n)a &\rightarrow b(OH^n)a& &\text{ and } & b(OH^n)a \otimes_\Z a(OH^n)a \rightarrow b(OH^n)a.
\end{align*}
It is not hard to see that those are defined using only cobordisms without split so that they do not depend on $C$. Moreover, by Proposition~\ref{prop:extalg}, $\bigoplus_{a\in B^n} a(OH^n_C)a$ is isomorphic to a direct sum of exterior algebras and thus, does not depend on $C$.
\end{proof}


\section{The odd center of $OH^n$ and the $(n,n)$-Springer variety}\label{sec:ospringer}

We are now ready to prove one of the main results of this paper, which is to construct an explicit isomorphism between the odd cohomology of the $(n,n)$-Springer variety and the odd center of the odd Khovanov arc algebra.

\begin{thm}\label{thm:iso}
There is an isomorphism of graded (super)algebras between $OZ(OH^n)$ and $OH(\mathfrak B_{n,n}, \Z)$. Moreover, this isomorphism is given by
$$h : OH(\mathfrak{B}_{n,n}, \Z) \rightarrow OZ(OH^n),\qquad x_i \mapsto \sum_{a \in B^n} a_i,$$
where $a_i$ is generated by the circle component of $W(a)a$ passing through the $i^{th}$ end point of~$a$, counting from the left.
\end{thm}

The proof of this theorem will occupy the rest of this section and is split into four steps. Firstly, we define a morphism $h_0 : OPol_{2n}  \rightarrow OZ(OH^n_C)$ and prove that $OI_n$ from Definition~\ref{def:oddcoh} lies in the kernel of this map, inducing the map $h$ on $OH(\mathfrak B_{n,n}, \Z)$. Secondly, we show that $h$ is injective using the equivalence up to sign between the odd and the even case together with Theorem~\ref{thm:isozhnhbnn}. Thirdly, we show that the ranks of the two algebras are equal using the cohomology of a geometric construction based on hypertori. Finally, we prove the theorem using all those ingredients.

\subsection*{Existence of $h$}

To construct $h$, we first define the algebra homomorphism
$$h_0 : OPol_{2n} \rightarrow OH^n_C,\qquad x_i \mapsto \sum_{a \in B^n} a_i,$$
where $a_i$ is generated by the circle component of $W(a)a$ passing through the $i^{th}$ end point of $a$, counting from the left.
It is well defined since 
$$h_0(x_ix_j) = \sum_{a \in B^n} a_i \wedge a_j = -\sum_{a \in B^n} a_j \wedge a_i = -h_0(x_jx_i).$$

\begin{lem}
The image of $h_0$ lies in the odd center of $OH^n_C$
$$h_0(OPol_{2n}) \subset OZ(OH^n).$$
\end{lem}

\begin{proof}
The proof is straightforward from Proposition~\ref{prop:caractoddcenter} and the fact that for all $a,b \in B^n$, we have
$$(h_0(x_i)) {_b1_a} = b_i .{_b1_a} = {_b1_a}.a_i = {_b1_a}(h_0(x_i)),$$
since $a_i$ and $b_i$ are both sent to the component of $W(b)a$ passing through the $i^{th}$ point.
\end{proof}

Now, we want to show that $\epsilon_r^S$ is in the kernel of $h_0$ for all $S$ and $r$ as in Definition~\ref{def:oddcoh}. 
This is equivalent to showing that $\epsilon_r^S$ lies in the kernel of the homomorphism
$$h_a : OPol_{2n} \rightarrow a(OH^n_C)a,\qquad x_i \mapsto a_i,$$ 
for all $a \in B^n$, since $h = \sum_{a \in B^n} h_a$. 

For the sake of simplicity, we fix an element $a \in B^n$. We also denote by $E_{2n}$ the ordered set $\{1, \dots, 2n\}$ and we see it as the set of end points of $a$, from left to right. For $S \subset E_{2n}$, we call a pair of distinct points in $S$ which are linked by an arc in $a$, an \emph{arc} of $S$. The other points of $S$ are called \emph{free points}. For $R = \{i_1, \dots, i_r\} \subset S$, we also write  
$$\epsilon_R^S := x_{i_1}^S \dots x_{i_r}^S$$
such that 
$$\epsilon_r^S = \sum_{R \subset S, |R| = r} \epsilon_R^S.$$

\begin{lem}\label{lem:freepoints}
Let $S \subset E_{2n}$ be a subset with $|S| = n+k$. Then, $S$ contains at least $k$ arcs and at most $n-k$ free points.
\end{lem}

\begin{proof}
There are at most $n$ free points and we pick $n+k$ points, thus we have to pick at least $k$ arcs. Then, there remain $n+k-2k$ points which can be free.
\end{proof}

\begin{lem}\label{lem:containsarc}
If $R \subset S \subset E_{2n}$ contains an arc of $S$, then 
$h_a(\epsilon_R^S) = 0.$
\end{lem}

\begin{proof}
This assertion follows from the fact that if $i$ is connected to $i'$ in $a$, then $a_i = a_{i'}$ and thus,
$h_a(x_i x_{i'}) = a_i \wedge a_{i'} = 0.$
\end{proof}

\begin{lem}\label{lem:rgen}
For all $R \subset S \subset E_{2n}$ with $|R| > n$, one has
$h_a(\epsilon_R^S) = 0.$
\end{lem}

\begin{proof}
There are at most $n$ free points in $S$, but $R$ contains at least $n+1$ points. So, $R$ contains an arc and the result follows from the preceding lemma. 
\end{proof}

\begin{lem}\label{lem:arcinR}
For all $R\subset S \subset E_{2n}$ with $|S| = n+k$ and $|R| \ge n - k +1$, there exists an arc $(j,j')$ in $S$ with $j$ or $j'$ in $R$. 
\end{lem}

\begin{proof}
We have to choose $n-k+1$ points in $S$, but there are at most $n-k$ free points by the Lemma~\ref{lem:freepoints}.
\end{proof}

\begin{example}\label{ex:hexists}
It is probably time to stop here a bit and look at an example that we will generalize below. So suppose
\begin{small}\begin{align*} 
a = 
\quad
&{\xy 
(0,2.5)*{1}; (5,2.5)*{2}; (10,2.5)*{3}; (15,2.5)*{4}; (20,2.5)*{5}; (25,2.5)*{6}; 
(30,2.5)*{7}; (35,2.5)*{8}; (40,2.5)*{9}; (45,2.5)*{10}; (50,2.5)*{11}; (55,2.5)*{12}; 
(0,0)*{};
(35,0)*{}
**\crv{ (0,-15) & (35,-15)};
(5,0)*{};
(10,0)*{}
**\crv{ (5,-3) & (10,-3)};
(15,0)*{};
(30,0)*{}
**\crv{ (15,-9) & (30,-9)};
(20,0)*{};
(25,0)*{}
**\crv{ (20, -3) & (25,-3)};
(40,0)*{};
(55,0)*{}
**\crv{ (40, -9) & (55,-9)};
(45,0)*{};
(50,0)*{}
**\crv{ (45, -3) & (50,-3)};
(-1,0)*{}; (56,0)*{} **\dir{--}; 
   \endxy}\ ,
\intertext{with $n=6$, $r=4$, $k=3$, $S = \{1,2,3,5,7,8,9,11,12\}$ and $R = \{5,8,9,11\} \subset S$. Putting a circle on the end points of $a$ which are in $S$, a bullet on the end points of $R$ and drawing the arcs that are not in $S$ by dotted lines we get}
&{\xy 
(0,5)*{\scriptstyle  1}; (5,5)*{2}; (10,5)*{3}; (15,5)*{4}; (20,5)*{5}; (25,5)*{6}; 
(30,5)*{7}; (35,5)*{8}; (40,5)*{9}; (45,5)*{10}; (50,5)*{11}; (55,5)*{12}; 
(0,0)*{\circ};
(35,0)*{\bullet}
**\crv{ (0,-15) & (35,-15)};
(5,0)*{\circ};
(10,0)*{\circ}
**\crv{(5,-3) & (10,-3)};
(15,0)*{ };
(30,0)*{ \circ}
**\crv{~*=<4pt>{.} (15,-9) & (30,-9)};
(20,0)*{\bullet};
(25,0)*{}
**\crv{ ~*=<4pt>{.}  (20, -3) & (25,-3)};
(40,0)*{\bullet};
(55,0)*{\circ}
**\crv{(40, -9) & (55,-9)};
(45,0)*{};
(50,0)*{\bullet}
**\crv{~*=<4pt>{.}  (45, -3) & (50,-3)};
   \endxy}\ .
\end{align*}\end{small}
The free points of $S$ are thus $\{5,7,11\}$. We have $\epsilon_R^S = (-x_5)(-x_8)x_9(-x_{11})$ and in $a(OH_C^n)a$ we obtain
$$
h_a(\epsilon^S_R) = (-A_4) \wedge (-A_1) \wedge A_5 \wedge (-A_6),
$$
where $A_i$ corresponds to the $i^{th}$ circle in $W(a)a$ counting from the left (see the paragraph about diagrammatic notation in Section~\ref{sect:hnodd}).
Now we look how $h_a(\epsilon^S_R)$ behaves when we modify $R$. Take $R' = \{5,8,11,12\}$, obtained from $R$ by exchanging the end points of the $5^{th}$ arc. Then we get $\epsilon_{R'}^S = (-x_5)(-x_8)(-x_{11})x_{12}$ and 
$$
h_a(\epsilon^S_{R'}) = (-A_4) \wedge (-A_1)  \wedge (-A_6) \wedge A_5 = -h_0(\epsilon^S_{R}).
$$
For $R'' = \{1,5,9,11\}$, obtained by replacing $8$ with $1$, we get $\epsilon_{R''}^S = x_1(-x_5)x_9(-x_{11})$ and 
$$
h_a(\epsilon^S_{R''}) = A_1 \wedge  (-A_4)   \wedge A_5 \wedge (-A_6) = h_0(\epsilon^S_{R}).
$$
Fortunately, by replacing $9$ with $12$ in $R''$ we get an element that will cancel with $h_0(\epsilon^S_{R''})$. If we remove $11$ or $9$ from $R$ and take $7$ instead, then exchanging $8$ with $1$ will also create two elements that cancel with each other. If we take $2$ or $3$ instead of $7$, then exchanging those points will also gives $0$. In general, the number of free points in $S$ which are not in $R$ plus the number of points in $R$ that are not free, both between the endpoints of an arc, will say if exchanging those endpoints change the sign in the image of $h_a$. We can observe that for each choice of $R$ there exists a possibility to exchange two points such that they cancel with each other in the image of $h_a$. At the end, we get $h_a(\epsilon_r^S) = 0$.
\end{example}

For $x \in S \subset E_{2n}$, we write 
$$F_S(x) := \begin{cases} 1, &\text{ if $x$ is free in $S$, } \\ 0, &\text{ otherwise,} \end{cases}$$
and, for $R \subset S$ and $(j,j')$ an arc of $S$ such that $j\in R$ or $j' \in R$, we define
\begin{align*}
p_{R,S}((j,j')) :&= \sum_{\substack{x \in S \setminus R \\ j < x < j'}} F_S(x)  + \sum_{\substack{y \in R \\ j < y < j'}} (1-F_S(y)) \mod 2\\
&\equiv \sum_{\substack{x \in S \\ j < x < j'}} F_S(x) + \sum_{\substack{y \in R \\ j < y < j'}} 1  \mod 2.
\end{align*}
We say that a point $x$ (resp. an arc $(k,k')$) belongs to an arc $(j,j')$ if $j < x < j'$ (resp. $j<k<k'<j'$). Therefore, $p_{R,S}((j,j'))$ counts the number of free points of $S$ belonging to the arc $(j,j')$ and which are not in $R$ plus the number of points of $R$ also belonging ton $(j,j')$ and which are not free. Denote the set of all sub-arcs of $(j,j')$ with an extremity in $R$ by $]j,j'[_R$. Also, write $]j,j[^{max}_R$ for the set of all maximal sub-arcs of $(j,j')$ with an extremity in $R$, that is, the sub-arcs not belonging to any other arcs from $]j,j'[_R$.

\begin{example}
Suppose we take $R$ as in Example~\ref{ex:hexists}. Then we get $p_{R,S}((1,8)) \equiv 3 \mod 2$ and $p_{R,S}((9,12)) \equiv 2 \mod 2$.
\end{example}

\begin{lem}\label{lem:subarccounting}
If no arc of $S \subset E_{2n}$ belongs to some $R \subset S$, then for any arc $(j,j')$ in $S$ with $j$ or $j'$ in $R$ we have
$$p_{R,S}((j,j')) = \sum_{\substack{(k,k') \in ]j,j'[_R^{max}}} (p_{R,S}((k,k'))+1) + \sum_{\substack{x \in S \setminus R \\ x \notin (k,k'), \forall (k,k') \in ]j,j'[_R}} F_S(x) \mod 2, $$
with the right sum on all $x \in S \setminus R$ which are not belonging to any sub-arcs of $]j,j'[_R$.
\end{lem}

\begin{proof}
All points belonging to a maximal sub-arc $(k,k')$ from the left sum belong to $(j,j')$. Moreover, $k$ or $k'$ is a free point in $R$ and thus, has a contribution by $+1$.
\end{proof}

\begin{lem}\label{lem:evenarc}
Let $R \subset S \subset E_{2n}$ be subsets with $|S| = n + k$ and $n-k+1 \le |R| \le n$. If no arc of $S$ belongs to $R$, then there exists an arc $(j,j')$ of $S$ with $j$ or $j'$ in $R$ and such that $p_{R,S}((j,j')) = 0 \mod 2$.
\end{lem}

\begin{proof}
First, by Lemma~\ref{lem:arcinR} there is at least one arc of $S$ with an extremity in $R$. We write $L \ne \emptyset$ for the set of all such arcs. Now, we suppose by contradiction that $p_{R,S} = 1 \mod 2$ on all those arcs. By the Lemma~\ref{lem:subarccounting}, we get for all $(j,j') \in L$,
$$p_{R,S}((j,j')) =  \sum_{\substack{x \in S \setminus R \\ x \notin (k,k'), \forall (k,k') \in ]j,j'[_R}}F_S(x) \mod 2,$$
with $x$ not belonging to any  sub-arc $(k,k') $ of $(j,j')$ such that $k$ or $k'$ is in $R$. Since by contradiction hypothesis this sum must be equal to $1 \mod 2$, there is at least one such $x$ which is free and thus at least $|L|$ free points in $S \setminus R$. However, we know that there are at least $n-k-|L|+1$ free points in $R$ and by Lemma~\ref{lem:freepoints} there are at most $n-k$ free points in $S$, which is a contradiction. 
\end{proof}

\begin{lem}\label{lem:signfromparity}
Let $R \subset S \subset E_{2n}$ be subsets and $(j,j')$ be an arc of $S$ with $j \in R$ (resp. $j' \in R$). We have 
$$h_a(\epsilon_{R}^S) = (-1)^{(p_{R,S}((j,j'))+1)} h_a(\epsilon_{R'}^S),$$
with $R'$ obtained by taking $R$ where we replace $j$ by $j'$ (resp. $j'$ by $j$), respecting the order of $S$.
\end{lem}

\begin{proof}
The proof is an induction on the size of $S \cap ]j,j'[$ and $R$. If $S \cap ]j,j'[ = \emptyset$, then clearly $p_R((j,j')) = 0 \mod 2$ and we get the result since $S(j') = S(j) +1$ and thus,
$x_{j}^S = -x_{j'}^S.$
Then,  to get the result in the general case, we check how the sign of the exterior product changes when we add free points and arcs to $S$ and points to $R$. We leave the details to the reader.
\end{proof}


\begin{prop}\label{prop:ocnkilled}
For all $a\in B^n$, $k \in \{1, \dots,n\}$, $r \ge n-k+1$ and all $S \subset E_{2n}$ such that $|S| = n+k$, we have
$$h_a(\epsilon_r^S) = 0.$$
\end{prop}

\begin{proof}
This result directly follows from Lemmas~\ref{lem:rgen},~\ref{lem:containsarc},~\ref{lem:evenarc}, and~\ref{lem:signfromparity}.
\end{proof}

\begin{cor}\label{cor:s}
The map $h_0$ induces a homomorphism of graded (super)algebras given by
$$h : OH(\mathfrak{B}_{n,n}, \Z) \rightarrow OZ(OH^n),\qquad x_i \mapsto \sum_{a \in B^n} a_i.$$
\end{cor}

\subsection*{Injectivity of $h$}

We will need the following result.

\begin{lem}
The algebra homomorphism induced by $h$
$$\overline h : OH(\mathfrak{B}_{n,n}, \Z) \otimes_\Z \Z/2\Z \rightarrow OZ(OH^n)  \otimes_\Z \Z/2\Z$$
 is an isomorphism.
\end{lem}

\begin{proof}
The proof results from the commutativity of the diagram
$$\xymatrix{
 OH(\mathfrak{B}_{n,n}, \Z) \otimes_\Z \Z/2\Z \ar[r]^{\overline h} \ar[d]_{\simeq} &  OZ(OH^n)  \otimes_\Z \Z/2\Z \ar[d]^{\simeq} \\
H(\mathfrak{B}_{n,n}, \Z)  \otimes_\Z \Z/2\Z \ar[r]_{\simeq}  & Z(H^n) \otimes_\Z \Z/2\Z
}$$
which comes from the equivalence modulo $2$ between the odd and the even cases, knowing that $H(\mathfrak B_{n,n}, \Z)$ is isomorphic to $Z(H^n)$ by a  morphism similar to $h$, see \cite[Section~5.3]{khovanov04} for more details.
\end{proof}

\begin{prop}\label{prop:injective}
The homomorphism $h$ is injective.
\end{prop}

\begin{proof}
From \cite[Theorem~3.8]{laudarussell14} we know that $OH(\mathfrak {B}_{n,n}, \Z)$ is a free $\Z$-module.
Hence the injectivity of $h$ follows from the lemma above and the fact that if a $\Z$-linear map between two free $\Z$-modules induces an isomorphism over $\Z/2\Z$ then it is injective.
\end{proof}

\subsection*{Computing the rank of $OZ(OH^n)$}

To show the existence of an isomorphism between $Z(H^n)$ and $H(\mathfrak B_{n,n})$, Khovanov constructed in \cite{khovanov04} a manifold $\widetilde S$ using products of $2$-spheres. We make a similar construction, but using circles instead of spheres. The inspiration from Khovanov's work should be clear. All cohomology groups and rings in this section are supposed to be taken on $\Z$.

\begin{definition}
For an $a \in B^n$, let $T_a \subset T^{2n} := \underset{2n}{\underbrace{S^1 \times \dots \times S^1}}$ be the set of all points $(x_1, \dots, x_{2n}) \in T^{2n}$ such that if $i$ is linked to $j$ by an arc of $a$, then $x_i = x_j$. We also define
$$\widetilde T := \bigcup_{a \in B^n} T_a \subset T^{2n}.$$
\end{definition}

One can notice that $T_a \simeq T^n$ as we equalize $n$ pairs of coordinates. In the same spirit, we have that $T_b \cap T_a \simeq T^{|W(b)a|}$, with $x_k = x_l$ whenever the $k^{th}$ and the $l^{th}$ end points are in the same component of $W(b)a$. As a result, $\widetilde T$ is a collection of hypertori identified to each other on certain subtori.

It is well-know that the cohomology ring of an $n$-torus is the exterior algebra generated by $n$ elements of degree $1$. If we forget the grading, then we get an isomorphism of superrings $a(OH^n_C)a \simeq H(T_a)$ and an isomorphism of abelian groups ${_b}(OH^n_C){_a} \simeq_{ab} H(T_b \cap T_a)$ for all $a,b \in B^n$. Lifting the (super)ring structure from $H(T_b \cap T_a)$, we get a (super)ring structure on ${_b}(OH^n_C){_a}$ and (super)ring morphisms
\begin{align*}
\gamma_{a; b,a} : a(OH^n_C)a \rightarrow b(OH^n_C)a,&\qquad x \mapsto {_b}1_ax, \\
\gamma_{b; b,a} : b(OH^n_C)b \rightarrow b(OH^n_C)a,&\qquad x \mapsto x{_b}{1_a}.
\end{align*}
The inclusions $T_b  \cap T_a \subset T_a$ and  $T_b  \cap T_a \subset T_b$ induce ring morphisms on the cohomology
\begin{align*}
\psi_{a; b,a} &: H(T_a) \rightarrow H(T_b \cap T_a), \\
\psi_{b; b,a} &: H(T_b) \rightarrow H(T_b \cap T_a). 
\end{align*}

\begin{lem}\label{lem:diaginclusions}
The morphisms defined above are such that the following diagram of ring morphisms commutes:
$$
\xymatrix{
H(T_b) \ar[r]^-{\psi_{b;b,a}} \ar[d]^{\simeq} & H(T_b \cap T_a) \ar[d]^{\simeq} & H(T_a) \ar[l]_-{\psi_{a;b,a}} \ar[d]^{\simeq} \\
b(OH^n_C)b \ar[r]_{\gamma_{b;b,a}} & b(OH^n_C)a & a(OH^n_C)a \ar[l]^{\gamma_{a;b,a}}.
}
$$
\end{lem}

\begin{proof}
Say $H(T_b) \simeq \Ext^* \{t_1, \dots, t_n\}$. Then the map $\psi_{b;a,b}$ identify $t_k$ with $t_l$ whenever the $k^{th}$ and the $l^{th}$ end points of $b$ are in the same component of $W(b)a$. The map $\gamma_{b;b,a}$ does exactly the same on the generators of $b(OH^n_C)b$ as $C_{bba}$ merges the corresponding components.
\end{proof}

\begin{definition}
Let $I$ and $J$ be finite sets and $A_i, B_j$ be rings for all $i \in I$ and $j \in J$. Moreover, let $\beta_{i,j} : A_i \rightarrow B_j$ be ring morphisms for some pairs $(i,j) \in I \times J$ with
 $$\beta := \sum \beta_{i,j} : \prod_{i \in I} A_i \rightarrow \prod_{j \in J} B_j.$$
We define the equalizer $\Eq(\beta)$ of $\beta$ as the subring of $\prod A_i$ such that for $(a_i)_{i\in I} \in \Eq(\beta)$ we have
$$\beta_{i,j} (a_i) = \beta_{k,j} (a_k)$$
whenever $\beta_{i,j}$ and $\beta_{k,j}$ are defined.
\end{definition}

By Proposition~\ref{prop:caractoddcenter}, we get that $OZ(OH^n) = \Eq(\gamma)$ for $\gamma := \sum_{a \ne b} \gamma_{a; b,a} + \gamma_{b; b,a}$. Thus, if we define $\psi := \sum_{a \ne b} \psi_{a;b,a} + \psi_{b; b,a}$, then by Lemma~\ref{lem:diaginclusions}, we get  a commutative diagram
\begin{align}
\xymatrix{
H(\widetilde T) \ar[r]^{\kappa} & \Eq(\psi) \ar[r]  \ar[d]^{\simeq} & \bigoplus\limits_{a \in B^n} H(T_a) \ar[r]^-{\psi} \ar[d]^-{\simeq} & \bigoplus\limits_{a \ne b \in B^n} H(T_b \cap T_a) \ar[d]^{\simeq} \\
OZ(OH^n) \ar[r]_{\simeq} & \Eq(\gamma) \ar[r] &   \bigoplus\limits_{a \in B^n} a(OH^n_C)a \ar[r]_-{\gamma} &  \bigoplus\limits_{a \ne b \in B^n} b(OH^n_C)a, 
}
\label{eq:tau}
\end{align}
with $\kappa$ coming from the factorization  by $\Eq(\psi)$ of the map $\phi : H(\widetilde T) \rightarrow \bigoplus_{a \in B^n} H(T_a)$ induced by the inclusions $T_a \hookrightarrow \widetilde T$. This factorization exists since $\img \phi \subset \Eq(\psi)$. Our goal now is to prove that $\kappa$ is an epimorphism such that $\rank( OZ(OH^n)) \le \rank(H(\widetilde T))$.

\begin{definition}
We say that there is an arrow $a \rightarrow b$ for $a,b \in B^n$ if there exists a quadruplet $1 \le i < j < k < l \le 2n$ such that $(i,j), (k,l) \in a$ and $(i,l), (j,k) \in b$. Visually, we have
$$\xy  (0,-1)*++{\Cmda} ;  (-7.5,2.5)*{i} ;  (-2.5,2.5)*{j} ;  (2.5,2.5)*{k} ;  (7.5,2.5)*{l} \endxy \longrightarrow \xy (-7.5,2.5)*{i} ;  (-2.5,2.5)*{j} ;  (2.5,2.5)*{k} ;  (7.5,2.5)*{l} ; (0,-2.5)*++{\Cmdb} \endxy.$$
This leads to a partial order $a \prec b$ if there exists a chain $a \rightarrow a_1 \rightarrow  \dots \rightarrow a_k \rightarrow b$. We extend (arbitrarily) this partial order to a total order $<$ on $B^n$.
\end{definition}

\begin{lem}
For all $a \in B_n$, we have
$$T_{<a} \cap T_a = \bigcup_{b \rightarrow a} (T_b \cap T_a)$$
with $T_{<a} := \bigcup_{b < a} T_b$.
\end{lem}

\begin{proof}
We use similar arguments as in \cite[Lemma~3.4]{khovanov04}, replacing $S$ by $T$.
\end{proof}

\begin{lem}\label{lem:celldecomp}
There exists a cellular decomposition of $T_a$ which restricts to a decomposition of $T_{<a} \cap T_a$, which itself restricts to the decomposition of $T_b \cap T_a$ for all $b \rightarrow a$. This decomposition is such that there are $\binom{n}{k}$ cells of dimension $k$ in $T_a$.
\end{lem}

\begin{proof}
We construct a similar decomposition as in \cite[Lemma~3.5]{khovanov04}. We stress the fact that the cells are not in even degree only for our case.
\end{proof}

\begin{cor}\label{cor:injdecomp}
The morphism
$$H(T_{<a} \cap T_a) \rightarrow \bigoplus_{b<a} H(T_b \cap T_a),$$
induced by the inclusions $(T_b \cap T_a) \subset (T_{<a} \cap T_a)$ is injective.
\end{cor}

We remark that $T_{\le a} = T_{<a} \cup T_a$ such that there is a Mayer-Vietoris sequence:
\begin{equation}
\xymatrix{
& \dots \ar[r] &  H^{m-1}(T_{a} \cap T_{<a}) \ar[dll]_{\delta}   \\
H^m(T_{a} \cup T_{<a})\ar[r]  &  H^m(T_{a}) \oplus H^m(T_{<a})\ar[r]  &  H^m(T_{a} \cap T_{<a})  \ar[dll]_\delta \\
  \ar[r] H^{m+1}(T_{a} \cup T_{<a})  & \dots
}\label{eq:mv}\end{equation}

\begin{prop}\label{prop:exactseq}
The following sequence is exact:
$$\xymatrix{
H(T_{\le a}) \ar[r]^-{\phi}  & \bigoplus_{b \le a} H(T_b) \ar[r]^-{\psi^-} & \bigoplus_{b < c \le a} H(T_b \cap T_c),
}$$ 
where $\phi$ is induced by the morphisms $T_b \hookrightarrow T_{\le a}$, and where we define
$$\psi^- := \sum_{b < c \le a} (\psi_{b,c} - \psi_{c,b}),$$
with
$$\psi_{b,c} = \psi_{b;b,c} : H(T_b) \rightarrow H(T_b \cap T_c),$$
induced by the inclusion $(T_b \cap T_c) \hookrightarrow T_b$.
\end{prop}

\begin{proof}
The proof is an induction on $a$ using Corollary~\ref{cor:injdecomp} like in \cite[Proposition~3.8]{khovanov04} with the only difference that we lose the left part $0 \rightarrow$ in our sequence.
\end{proof}

\begin{prop}\label{prop:epi}
There is an epimorphism of superrings 
$$k : H(\widetilde T) \rightarrow OZ(OH^n).$$
\end{prop}

\begin{proof}
We take $a$ maximal in Proposition~\ref{prop:exactseq} giving an exact sequence
$$\xymatrix{
H(\widetilde T) \ar[r]^-{\phi}  & \bigoplus_{b} H(T_b) \ar[r]^-{\psi^-} & \bigoplus_{b < c} H(T_b \cap T_c)
}$$ 
and we observe that by definition
$$\Eq(\psi) = \ker(\psi^-) = \img(\phi)$$
so that $\kappa : H(\widetilde T) \rightarrow \Eq(\psi)$ from diagram (\ref{eq:tau}) is surjective. 
\end{proof}

Now we show that the rank of $H(\widetilde T)$ is the same as the rank of $OH(\mathfrak B_{n,n}, \Z)$. By \cite[Corollary~3.9]{laudarussell14}, we already know that  $\rank(OH(\mathfrak B_{n,n}, \Z)) = \binom{2n}{n}$.

%

\begin{lem}\label{lem:hom}
For all $k \ge 0$, the cohomology groups
\begin{align*}
H^k(T_a \cup T_{<a}), H^k(T_a), H^k(T_{<a}) \text{ and } H^k(T_a \cap T_{<a})
\end{align*}
are free of ranks satisfying the relation
$$\rank(H^k(T_a \cup T_{<a})) = \rank(H^k(T_a)) + \rank(H^k(T_{<a})) - \rank(H^k(T_a \cap T_{<a})).$$
\end{lem}

\begin{proof}
The proof is an induction on $B^n$. We consider the Mayer-Vietoris sequence (\ref{eq:mv})
and we claim that the morphisms $H^k(T_a) \rightarrow H^k(T_a \cap T_{<a}) $ are surjective and thus, that the boundary operators $\delta$ are zero. Indeed, the decomposition from Lemma~\ref{lem:celldecomp} has the same number of $k$-cells as the rank of $H^k(T_a)$ so characteristic functions on them are independent generators for the cohomology. Seeing that the cell decomposition restricts to $T_a \cap T_{<a}$, the cohomology groups $H^k(T_a \cap T_{<a})$ are free of rank given by the cell decomposition and the morphisms are surjective. 
This claim gives us an isomorphism
$$H^k(T_a \cap T_{<a}) \simeq \frac{H^k(T_a) \oplus H^k(T_{<a})}{H^k(T_a \cup T_{<a})}.$$
If $a$ is minimal, then the lemma is trivial. For the general case, we get the result by induction since $H^k(T_a)$, $H^k(T_{<a})$ and $H^k(T_a \cap T_{<a})$ are free and so is $H^k(T_a \cup T_{<a})$.
\end{proof}

\begin{prop}\label{prop:grouprank}
$H(\widetilde T)$ is a free abelian group of rank
$$\rank(H(\widetilde T)) = \binom{2n}{n}.$$
\end{prop}

\begin{proof}
We obtain a cellular partition of $\widetilde T$ by first taking the cellular decomposition of $T_{a_0}$ from Lemma~\ref{lem:celldecomp}, with $a_0 \in B^n$ the minimal element, and then by adding the cells $T_{a_m} \setminus T_{<a_m}$ for all $a_m \in B^n$ following the total order. We claim that the rank of $H(\widetilde T)$ is given by the number of cells of the partition. 
Indeed, all cohomology groups are free and the relation from the Lemma~\ref{lem:hom} gives us the claim since $\rank(H(T_{a_m})) - \rank(H(T_{a_m} \cap T_{<a_m}))$ counts exactly the number of cells of $T_{a_m} \setminus T_{<a_m}$.
Finally, like in \cite{khovanov04} (and proved in \cite[Lemma~3.64]{naisse15}), the number of cells is $\binom{2n}{n}$ and this concludes the proof.
\end{proof}

\begin{cor}\label{cor:ranks}
$OZ(OH^n)$ is a free abelian group and we have
$$\rank(OZ(OH^n)) = \rank(OH(\mathfrak B_{n,n}, \Z)).$$
\end{cor}

\begin{proof}
By Proposition~\ref{prop:injective}, we have
$$\rank(OZ(OH^n)) \ge \rank(OH(\mathfrak B_{n,n}, \Z))$$
and by Propositions~\ref{prop:epi}~and~\ref{prop:grouprank}, and~\cite[Corollary~3.9]{laudarussell14}, we get
$$\rank(OZ(OH^n)) \le \rank(H(\widetilde T)) =  \binom{2n}{n} =  \rank(OH(\mathfrak B_{n,n}, \Z)).$$
The two inequalities together conclude the proof.
\end{proof}

\subsection*{Proof of Theorem \ref{thm:iso}}

In order to prove Theorem \ref{thm:iso}, we construct a surjective map $\pi : OH(\mathfrak B_{n,n}, \Z) \rightarrow H(\widetilde T)$ such that the following diagram commutes:
\[
\xymatrix{
OH(\mathfrak B_{n,n}, \Z)   \ar[rr]^{h} \ar[rd]_\pi && OZ(OH^n) \\
&H(\widetilde T). \ar[ur]_{k} &
}
\]

Following \cite[Section~4]{khovanov04}, let $\imath^* : H(T ^{2n}) \rightarrow H(\widetilde T)$ be the homomorphism induced by the inclusion $\widetilde T \subset T^{2n}$. 

\begin{lem}
The map $\imath^* : H(T ^{2n}) \twoheadrightarrow H(\widetilde T)$ is an epimorphism.
\end{lem}

\begin{proof}
The map induced on the homology $H_1(\widetilde T) \hookrightarrow H_1(T^{2n})$ is split injective, and both groups are free. Indeed the cellular decomposition of $\widetilde T$ has all 1-dimensional cells given by linearly independent diagonals in $T^{2n}$, which can be completed into a basis for $H_1(T^{2n})$. Hence, by the universal coefficient theorem the map induced on the cohomology $H^1(T^{2n}) \twoheadrightarrow H^1(\widetilde T)$ is surjective. 
The same applies for all $H^k(T^{2n}) \rightarrow H^k(\widetilde T)$, with linearly independent hyperplanes of dimension $k$. Indeed cells of $T_a \setminus T_{<a}$ are given by hyperplanes with at least one generating vector linearly independent to the cells of $T_{<a}$. This concludes the proof.
%
%
\end{proof}


 We write $X_i \in H(\widetilde T)$ for the image by $\imath^*$ of the generating element in the cohomology of $T^{2n}$ given by the characteristic function on the $i^{th}$ component. In other words, if we say $p_i$ is the projection $T^{2n} \rightarrow S^1$ onto the $i^{th}$ component, then $X_i = i^* \circ p_i^* (X)$, where $H(S^1) =  \Ext^* \{X\}$. 

Then we define the ring homomorphism $\pi_0 : OPol_{2n} \rightarrow H(\widetilde T)$ by $\pi_0(x_i) = X_i$. Clearly we have an isomorphism $H(T^{2n}) \cong \Ext^*\{p_1^*(X), \dots, p_{2n}^*(X)\}$, and so $\pi_0$ is surjective by the lemma above.

\begin{lem}
The map $k : H(\widetilde T)  \rightarrow OZ(OH^n)$ is an isomorphism.
\end{lem}

\begin{proof}
It is an epimorphism between free abelian groups of the same rank.
\end{proof}

From this, we deduce the homomorphism $j^* : H(\widetilde T) \rightarrow \bigoplus_{a \in B^n} H(T_a) \cong \bigoplus_{a \in B^n} a(OH^n_C)a$, induced by the inclusions $T_a \subset \widetilde T$, is an injection.
Moreover, by construction of $\pi_0$, the diagram
\[
\xymatrix{
OPol_{2n} \ar[rr]^{h_0} \ar[rd]_{\pi_0} && \bigoplus_{a \in B^n} a(OH^n_C)a \\
&H(\widetilde T) \ar[ur]_{j^*} &
}
\]
commutes.

\begin{lem}
We have $\pi_0(OC_n) = 0$.
\end{lem}

\begin{proof}
By the commutativity of the diagram above, we get $j^* \circ \pi_0(OC_n) = h_0(OC_n)$, which is zero by Proposition~\ref{prop:ocnkilled}. This concludes the proof since $j^*$ is injective.
\end{proof}

Therefore, there is an induced map $\pi : OH(\mathfrak B_{n,n}, \Z)  \rightarrow H(\widetilde T)$ given by $\pi(x_i) = X_i$. Since $\pi_0$ is surjective, so is $\pi$. 

\begin{proof}[Proof of Theorem \ref{thm:iso}]
By construction, we have $h = k \circ \pi$. Since $k$ and $\pi$ are both surjective, $h$ is surjective as well. By Proposition~\ref{prop:injective}, it is also injective and thus an isomorphism.
\end{proof}

\section{Turning $OH^n_C$ into an associative algebra}\label{sec:assoc}

In this section, we show that we can twist the multiplication of $OH^n_C$, turning it into an associative $\Z[i]$-algebra. To do so, we begin by proving that $OH^n_C$ is a quasialgebra in the sense of Albuquerque-Majid \cite{octonions99}, when graded by a groupoid as in~\cite{putyrapreprint}. Finally, we give some classification results on all those algebras.

\subsection{The Putyra-Shumakovitch associator}

The material in this subsection is due to Putyra and Shumakovitch \cite{putyrashumakovitch} \footnote{And we would like to thank Krzysztof Putyra for explaining it to us. }.

\paragraph{Grading by a groupoid} A groupoid is a small category with every morphism admitting an inverse. We say that a ring $R$ is graded by a groupoid $\mathcal G$ if 
\begin{align*}
R &= \bigoplus_{g \in \Hom(\mathcal G)} R_g& & \text{and}& 
R_{g_1} R_{g_2} &\subset R_{g_1 \circ g_2},\end{align*}
whenever $g_1$ and $g_2$ are composable, and $R_{g_1} R_{g_2} = 0$ otherwise.

\paragraph{Arc grading}
Let $\mathcal G^n$ be the groupoid with objects given by the elements of $B^n$ and with a unique morphism $a \rightarrow b$ for all $a,b \in B^n$. By uniqueness of the morphisms, the composition is such that $a \rightarrow b \rightarrow c$ is equal to $a \rightarrow c$. We can view the morphism $a \rightarrow b$ as the diagram $W(b)a$ with the composition defined for all $a,b,c \in B^n$ by $W(c)b \circ W(b)a = W(c)a$. It is clear that $\mathcal G^n$ is a groupoid as every morphism $a \rightarrow b$ possesses an inverse $b \rightarrow a$ and $a \rightarrow a$ is the identity.

\begin{example}
We  can put $\mathcal G^n$ in form of a diagram. For example, $\mathcal G^2$ can be pictured as

$$
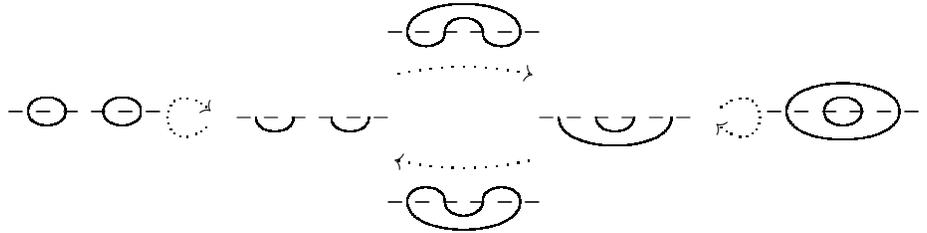
\xymatrix@H=5pc@C=0.1pc{
 \ar@{.>}@(dl,ul)^{\xy  (0,0)*{\CmdWa}; (0,0)*{\Cmda}; (10,0)*{} \endxy}  & {\Cmda} \ar@{.>}@(ur,ul)[rrrrrrrr]^{\xy  (0,2)*{\CmdWb}; (0,2)*{\Cmda} \endxy}  &&&&&&&& {\Cmdb} \ar@{.>}@(dl,dr)[llllllll]^{\xy  (0,-5)*{\CmdWa}; (0,-5)*{\Cmdb} \endxy} & \ar@{.>}@(ur,dr)^{\xy  (0,0)*{\CmdWb}; (0,0)*{\Cmdb}; (-10,0)*{} \endxy} 
}.$$
\end{example}

The decomposition 
$$OH^n = \bigoplus_{a,b \in B^n} b(OH^n)a =  \bigoplus_{W(b)a \in \Hom(\mathcal G^n)} b(OH^n)a $$
 gives a grading of $OH^n_C$ by $\mathcal G^n$. 
The integer degree, coming from the grading of $\Ext^* V(S)$, will be called the \emph{quantum degree} and written $|\_|_q$. 
The degree coming from the grading of $\mathcal G^n$ will be called the \emph{arc degree} and denoted $|\_|_B$.
We get a bigrading, written $|\_|$, by the groupoid $\mathcal G^n \times \mathcal Z$, with $\mathcal Z$ being the abelian group $\Z$ viewed as a category with one abstract object $\star$, morphisms given by the integers and composition obtained by taking the sum, i.e.
$$\star \xrightarrow{z_1} \star \xrightarrow{z_2} \star = \star \xrightarrow{z_1 + z_2} \star.$$
Notice that $\Hom(\mathcal G^n \times \mathcal Z) \simeq \Hom(\mathcal G^n) \times \Z$.

\medskip

As a matter of fact, $OH^n_C$ is graded by a subgroupoid with arrows given by diagrams $W(b)a$ and a quantum degree in $2\Z$ or $2\Z+1$ depending on whether $|W(b)a| \equiv n \mod 2$ or~not. 

\begin{definition}\label{def:subgroupoid}
We denote by $\mathcal G^n \times \mathcal Z_2$ the groupoid given by the same objects as $\mathcal G^n \times \mathcal Z$ but with hom-spaces defined as 
\[ 
\hom((a, \star), (b, \star)) = \left\{ (W(b)a, n) \left| \begin{array}{ll} n \in 2\Z &\text{ if } |W(b)a| \equiv n \mod 2, \\ n \in 2\Z + 1 &\text{ otherwise.}\end{array}\right. \right\}.
\]
\end{definition}

\paragraph{Quasialgebras} Recall that we proved in Proposition~\ref{prop:nonassoc} that $OH^n_C$ is not an associative algebra. We claim that it is almost one: it is a quasialgebra in the sense of \cite{octonions99} but graded by a groupoid  as in \cite{putyrapreprint}.
Before defining quasialgebras, recall that the \emph{nerve} of a category $\mathcal  C$ is the simplicial set generated by its morphisms. We denote it $N(\mathcal C)$ and we denote by $N_n(\mathcal C)$ the set of compositions of $n$ morphisms in~$\mathcal C$.

\begin{definition}
A \emph{quasialgebra} $A$ is a (nonassociative) $R$-algebra graded by a groupoid $\mathcal G$ with a 3-cocycle
$$\phi : N_3(\mathcal G)  \rightarrow R^*$$
where $R^* \subset R$ are the invertible elements, such that
\begin{equation}\label{eq:assoc}(xy)z = \phi(|x|,|y|,|z|)x(yz)\end{equation}
for all (homogeneous with compatible degrees) $x,y,z \in A$. We call $\phi$ the \emph{associator} of $A$.
\end{definition}

\begin{rem}\label{rem:assoc}
The condition $\phi$ being a 3-cocycle means
$$ \phi(h,k,l) \phi(g,hk,l) \phi(g,h,k) = \phi(gh,k,l)\phi(g,h,kl)$$
for all sequences
$$ e \xleftarrow{g} d \xleftarrow{h} c \xleftarrow{k} b \xleftarrow{l} a \in \mathcal G.$$
We also require $\phi(h,\id_b, l) = 1$ whenever $\id_b$ is an identity morphism.
Hence the following diagram commutes:
\begin{align} \label{diag:pentagon}
\xymatrix{
((A \otimes A ) \otimes A ) \otimes A  \ar[rr]^{\phi \otimes \id} \ar[d]_{\phi} && (A\otimes (A \otimes A))\otimes A \ar[d]^{\phi} \\
(A\otimes A) \otimes (A\otimes A) \ar[dr]_{\phi} && A\otimes((A\otimes A)\otimes A) \ar[dl]^{\id\otimes\phi} \\
&A\otimes(A\otimes(A\otimes A)). &
}
\end{align}
Notice if we have a non-associative $R$-algebra $A$ graded by a groupoid $\mathcal G$ and if a $\phi : N_3(\mathcal G)  \rightarrow R^*$ respecting~(\ref{eq:assoc}) exists, then its definition on $\deg A := \{ \deg(x) \in \Hom(\mathcal G) | x \in A\}$ is forced by~(\ref{eq:assoc}). Moreover, it will be a  3-cocycle on this subgroupoid since (\ref{diag:pentagon}) must commute for the multiplication in $A$ to be well-defined.
\end{rem}

We can view $\Z^* = \{\pm 1\}$ as the group $\Z/2\Z$ by the isomorphism $x \in \Z/2\Z \mapsto (-1)^x$. Because of that, we will write the associator of $OH^n_C$ as a map with codomain $\Z/2\Z$.

\begin{lem}\label{lem:chchassoc}
There exists a unique map
$$\varphi_C^{ch} : N_3(\mathcal G^n) \rightarrow \Z/2\Z$$
such that for all $a,b,c,d \in B^n$ we have
$$OF(C_{dba} \circ C_{dcb}\Id_{W(b)a}) = (-1)^{\varphi_C^{ch}(W(d)c, W(c)b, W(b)a)} OF(C_{dca} \circ \Id_{W(d)c}C_{cba}).$$
\end{lem}

\begin{proof}
The two cobordisms $C_{dba} \circ C_{dcb}\Id_{W(b)a}$ and $C_{dca} \circ \Id_{W(d)c}C_{cba}$ have the same Euler characteristic and the same source and target, and thus are homeomorphic. This means they are related by changes of chronology and orientations, which induce only potential changes of sign.
\end{proof}

\begin{lem}\label{lem:nbrsplits}
For all choices $C\in \mathcal C^n$, the cobordism $C_{cba}$ is composed by the same number of splits.
\end{lem}

\begin{proof}
This is immediate as all $C_{cba}$ have the same Euler characteristic.
\end{proof}

We define 
\begin{align*}
\varphi^{com} : N_3(\mathcal G^n \times \mathcal Z) \subset (\Hom(\mathcal G^n) \times \Z)^3 &\rightarrow \Z/4\Z, \\
\left( (W(d)c, k), (W(c)b, l), (W(b)a, m) \right) &\mapsto
\left(k- n +|W(d)c|\right) s(W(c)b, W(b)a) \mod 4, 
\end{align*}
where $s(W(c)b, W(b)a)$ is the number of splits coming from the cobordism $C_{cba}$, which does not depend on $C$ by Lemma~\ref{lem:nbrsplits}. If we take $x,y,z \in OH^n_C$, then
\begin{equation}\label{eq:phicom}
\varphi^{com}(|x|,|y|,|z|) = 2p(x)s(|y|_B, |z|_B).
\end{equation}
In this spirit, we define the parity $p((W(d)c, k)) := \frac{k-n+|W(d)c|}{2} \in \{0,1/2,1,3/2\}$ for any element in $\Hom(\mathcal G^n \times \mathcal Z)$. Notice that the parity gives an integer for every element in the subgroupoid $\mathcal G^n \times \mathcal Z_2$ from Definition~\ref{def:subgroupoid}.

\begin{lem}\label{lem:assoc}
The map
$$\varphi_C := \varphi_C^{ch} + \varphi^{com}/2 :  N_3(\mathcal G^n \times \mathcal Z_2) \rightarrow \Z/2\Z$$
is such that
$$(xy)z = (-1)^{\varphi_C(|x|,|y|,|z|)} x(yz)$$
for all $x,y,z \in OH^n_C$.
\end{lem}

\begin{proof}
Suppose we have $x \in d(OH^n_C)c$, $y \in c(OH^n_C)b$ and $z \in b(OH^n_C)a$. We compute
\begin{align*}
xy &= S(d,c,b) \wedge x  \wedge y, \\
(xy)z &= S(d,b,a) \wedge S(d,c,b) \wedge x \wedge y \wedge z, \\
yz &= S(c,b,a) \wedge y \wedge z, \\
x(yz) &= S(d,c,a) \wedge x \wedge S(c,b,a) \wedge y \wedge z,
\end{align*}
where $S(d,c,b)$ are the terms coming from the splits of the cobordism $C_{dcb}$. Notice that we abuse the notation by identifying $x,y$ with their images in $d(OH^n_C)b$ in the first line, and so on. \\
This computation means that the non-associativity comes from two phenomena:
\begin{itemize}
\item The commutation between the elements coming from the splits of the product $yz$ and the left term $x$, that is
$$S(d,c,a) \wedge x \wedge S(c,b,a) \wedge y \wedge z = (-1)^{p(x)p(S(c,b,a))} S(d,c,a) \wedge S(c,b,a) \wedge x \wedge y \wedge z.$$
By (\ref{eq:phicom}), we have $p(x)p(S(c,b,a)) = p(x)s(W(c)b,W(b)a) = \varphi^{com}(|x|,|y|,|z|)/2$.
\item The change of chronology and orientations between the cobordisms $C_{dba} \circ C_{dcb}\Id_{W(b)a}$ and $C_{dca} \circ \Id_{W(d)c}C_{cba}$, meaning that
$$S(d,b,a) \wedge S(d,c,b) = (-1)^{\varphi_C^{ch}(|x|_B,|y|_B, |z|_B)} S(d,c,a) \wedge S(c,b,a)$$
by Lemma~\ref{lem:chchassoc}.
\end{itemize}
To conclude, we have
$$(xy)z = (-1)^{\varphi_C^{ch}(|x|_B,|y|_B, |z|_B)+\varphi^{com}(|x|,|y|,|z|)/2} x(yz)$$
and this finishes the proof.
\end{proof}
%

\begin{lem}\label{lem:cocycle}
The map $\varphi_C$ is a $3$-cocycle. More generally, the map
$$\psi_C := 2 \varphi_C^{ch} + \varphi^{com} :  N_3(\mathcal G^n \times \mathcal Z) \rightarrow \Z/4\Z,$$
is a $3$-cocycle.
\end{lem}

\begin{proof}
We mainly use Remark \ref{rem:assoc}. Take $a,b,c,d,e \in B^n$. Substituting ${_e1_d}, {_d1_c}, {_c1_b}$ and ${_b1_a}$ in~(\ref{diag:pentagon}), we get
\begin{equation}\label{eq:cocyclech}
ch_{edcb} + ch_{edba} + ch_{dcba} = s_{edc}s_{cba} + ch_{ecba} + ch_{edca}
\end{equation}
where $s_{cba} = s(W(c)b, W(b)a)$, $ch_{dcba} = \phi_C^{ch}(W(d)c,W(c)b,W(b)a)$, and so on. Now suppose we have a sequence $(e,\star) \xleftarrow{g} (d,\star) \xleftarrow{h} (c, \star) \xleftarrow{k} (b,\star) \xleftarrow{l} (a,\star)  \in \mathcal G^n \times \mathcal Z$. We compute
\begin{align*}
\psi_C(h,k,l) + \psi_C(g,hk,l) + \psi_C(g,h,k) &=\\ 
2(p(h)s_{cba} &+ ch_{dcba} + p(g)s_{dba} + ch_{edba} + p(g)s_{dcb} + ch_{edcb}), 
\intertext{and}
\psi_C(gh,k,l) + \psi_C(g,h,kl) &= 2\left(p(gh)s_{cba} + ch_{ecba} + p(g)s_{dca} + ch_{edca}\right).
\end{align*}
It is easy to see that $p(gh) = p(g) + p(h) + s_{edc}$ and  
$$s_{dba}+s_{dcb} = s_{cba}+s_{dca} =  \#\text{splits in the cobordism } W(d)cW(c)bW(b)a \rightarrow W(d)a,$$ 
such that by (\ref{eq:cocyclech}) we get 
$$\psi_C(h,k,l) + \psi_C(g,hk,l) + \psi_C(g,h,k) = \psi_C(gh,k,l) + \psi_C(g,h,kl),$$
which  concludes the proof for $\psi_C$. We get the claim for $\varphi_C$ by seeing that $\varphi_C = \psi_C |_{\mathcal G \times \mathcal Z_2}/2$.
\end{proof}

\begin{thm}
The nonassociative ring $OH^n_C$ is a quasialgebra with associator $\varphi_C$.
\end{thm}

\begin{proof}
This is an immediate consequence from Lemmas~\ref{lem:assoc}~and~\ref{lem:cocycle}.
\end{proof}

We call $\varphi_C$ the Putyra-Shumakovitch associator.

\subsection{Twisting $OH^n_C$}

\paragraph{Twisted multiplication} 
The idea of twisting a $\mathcal G$-graded $R$-algebra $A$ by a map
$$\tau : N_2(\mathcal G) \rightarrow R^*$$
is to define a new algebra $A_\tau$ by the same elements as $A$, but with a multiplication given by
$$A_\tau \otimes_R A_\tau \rightarrow A_\tau, \quad (x,y) \mapsto x *_{\tau} y := \tau(|x|,|y|) xy$$
for all $x,y \in A_\tau$, and where $xy$ is the product in $A$.

\begin{prop}\label{prop:cobtwist}
Let $A$ be a quasialgebra graded by $\mathcal G$ with associator $\phi$. If $\phi$ is a coboundary, then there exist a twist $\tau$ such that $A_\tau$ is associative.
\end{prop}

\begin{proof}
By definition of coboundary, there exists a map
$$\tau : N_2(\mathcal G) \rightarrow R^*$$
such that
\begin{equation}\label{eq:2cocycle}\phi(g,h,k) = \tau(g,h)\tau(g,hk)^{-1}\tau(gh,k)\tau(h,k)^{-1}\end{equation}
for all sequence $ d \xleftarrow{g} c \xleftarrow{h} b \xleftarrow{k} a \in \mathcal G$. Let $A_\tau$ be the twisting of $A$ by this $\tau$. Then we have
\begin{align*}
(x *_\tau y)*_\tau z &= \tau(|x|,|y|)\tau(|xy|,|z|) (xy)z, \\
x*_\tau (y*_\tau z) &= \tau(|x|,|yz|)\tau(|y|,|z|) x(yz),
\end{align*}
for all $x,y,z \in A_\tau$ and thus, by (\ref{eq:assoc}) and (\ref{eq:2cocycle}), we conclude that $A_\tau$ is associative.
\end{proof}

The geometric realization of the nerve of a category $\mathcal C$, denoted $|N(\mathcal C)|$, is a topological space constructed by gluing simplexes respecting the simplicial structure of the nerve.

\begin{lem}\label{lem:geomrel}
The geometric realization of $N(\mathcal G^n)$ is a simplex of dimension $C_n-1$, for $C_n$ the $n^{th}$ Catalan number:
$$|N(\mathcal G^n)| \simeq \Delta^{(C_n-1)}.$$
\end{lem}

\begin{proof}
The proof is immediate from the fact that $B^n$ has cardinality $C_n$ and $\mathcal G^n$ possesses one unique morphism between each pair of objects.
\end{proof}

\begin{lem}\label{lem:coh}
The cohomology groups of $\mathcal{G}^n \times \mathcal{Z}$ are
\begin{align*}
H^0(N(\mathcal{G}^n \times \mathcal{Z}), \Z/4\Z) &\simeq \Z/4\Z, \\
H^1(N(\mathcal{G}^n \times \mathcal{Z}), \Z/4\Z) &\simeq \Z/4\Z , \\
H^{\ge 2}(N(\mathcal{G}^n \times \mathcal{Z}), \Z/4\Z) &\simeq 0.
\end{align*}
\end{lem}

\begin{proof}
First, by Lemma~\ref{lem:geomrel}, we get that $|N(\mathcal G^n \times \mathcal Z)| \simeq  \Delta^m \times S^1$ for $m = C_n -1$. By the K\"unneth formula, we get 
$$H^k(N(\mathcal G^n \times \mathcal Z), \Z/4\Z) \simeq \bigoplus_{i+j=k} H^i(\Delta^{m}, \Z/4\Z) \otimes_{\Z/4\Z} H^j( S^1, \Z/4\Z),$$
which proves the claim.
\end{proof}

Some technicalities still remain to be solved, before being able to apply Proposition~\ref{prop:cobtwist} to $OH^n_C$: we do not know the cohomology of $\mathcal G^n \times \mathcal Z_2$ and thus, we are not able to show that the Putyra-Shumakovitch associator is a coboundary. Therefore, we work with $\psi_C$, which has an image in $\Z/4\Z$ and gives square roots of $-1$. Hence, we must consider the extended version $OH^n_C \otimes_\Z \Z[i]$ to the Gaussian integers, with $\Z[i]^* = \{1,i,-1,-i\} \simeq \Z/4\Z$. By Lemmas~\ref{lem:assoc}~and~\ref{lem:cocycle}, $OH^n_C \otimes_\Z \Z[i]$ is a quasialgebra graded by $\mathcal G^n \times \mathcal Z$ with $\psi_C$ as associator.

\begin{thm}
For all $C \in \mathcal C^n$ there exists a map $\tau_C :N_2(\mathcal G^2 \times \mathcal Z) \rightarrow \Z/4\Z$ such that the twisted algebra $ \left(OH^n_C \otimes_\Z \Z[i] \right)_{\tau_C}$ is associative.
\end{thm}

\begin{proof}
By Lemma \ref{lem:coh}, we have 
$$H^3(N(\mathcal{G}^n \times \mathcal{Z}), \Z/4\Z) \simeq 0,$$
 and thus every $3$-cocycle is a coboundary. In particular, the associator $\psi_C$ is a coboundary and we can apply Proposition \ref{prop:cobtwist}.
\end{proof}

\begin{rem}
Notice that the twist is not necessarily unique and thus we get potentially a family of associative algebras for each $C \in \mathcal C^n$.
\end{rem}

\begin{example}\label{ex:twistedoh2}
We construct an explicit $\left(OH^2_C\right)_{\tau_C}$ based on the choice $C$ from Remark~\ref{rem:standardchoice}. 
We twist this algebra with 
\begin{align*}
\tau_C\left((\aoat, 1+4k), (\yo,*)\right) &= i,& \tau_C\left((\yo,2+4k), (\xo,*)\right) = &i, \\
\tau_C\left((\aoat, 2+4k), (\yo,*)\right) &= -1,& \tau_C\left((\yo,3+4k), (\xo,*)\right) = &-1, \\
\tau_C\left((\aoat, 3+4k), (\yo,*)\right) &= -i,& \tau_C\left((\yo,0+4k), (\xo,*)\right) = &-i,
\end{align*}
for every $k \in \Z$ and $\tau_C = 1$ everywhere else. In short, we get the multiplication table:
\begin{center}
\begin{tabular}{|c"c|c|c|c|c|c|}
  \hline
 $\left(OH^2_C\right)_{\tau_C}$ & $\idaa$ & $\aoid$ & $\idat$ & $\aoat$ & $\idab$ & $\xo$ \\
  \thickhline
$\idaa$ & $\idaa$ & $\aoid$ & $\idat$ & $\aoat$ &  $\idab$ & $\xo$  \\
  \hline
$\aoid$ & $\aoid$ & $0$ & $\aoat$ & $0$  & $\color{red}- \xo$ & $0$  \\
  \hline
$\idat$ & $\idat$ & $-\aoat$ & $0$ & $0$ & $\color{red}- \xo$ & $0$ \\
  \hline
$\aoat$ & $\aoat$ & 0 & 0 & 0 & 0 & 0  \\
  \hline
${\idba}$ & ${\idba}$ & $\yo$ & $\yo$ & 0 & $\idbt - \boid$ & $-\bobt$  \\
  \hline
$\yo$ & $\yo$ & 0 & 0 & 0 & $\color{red} \bobt$ & 0 \\
  \hline
\end{tabular}
\end{center}
for $a$ and the one for $b$ stays the same. An exhaustive computation (which can easily be done by computer) confirms that $d\tau_C$ gives the associator in this case.
\end{example}

\begin{rem}
For this example, the twisting in $OH^n_C$ results in integer coefficients which is not surprising since the geometric realization of $\mathcal G^2 \times \mathcal Z_2$ has dimension 2 and thus there exists a twist for the associator $\varphi_C$ in this case. In general, this is not true and we lose the property that the algebra agrees modulo 2 with $H^n \otimes_\Z \Z[i]$.
\end{rem}

\subsection{Classification results}

For now, we have a family of quasialgebras $\{OH^n_C\}$ indexed by $\mathcal C^n$ and a family of associative algebras indexed by $\mathcal C^n$ and the twists. In this section, we partially classify these families.

\begin{prop}\label{prop:classodd}
Let $C, C'$ be two choices in $\mathcal C^n$ and $\varphi_C,\varphi_{C'}$ be respectively the associators of $OH^n_C$ and $OH^n_{C'}$. If $\varphi_C = \varphi_{C'}$, then the two quasialgebras are isomorphic, $OH^n_C \simeq OH^n_{C'}$.
\end{prop}

\begin{proof}
Seeing that $C_{cba}$ and $C'_{cba}$ are related by a change of chronology and orientations, there is a map $\eta_{C,C'} :N_2(\mathcal G^n) \rightarrow \Z/2\Z$ such that
$$OF(C_{cba}) = (-1)^{\eta_{C,C'}(W(c)b, W(b)a)} OF(C'_{cba}),$$
for all $a,b,c \in B^n$. Writing $*_C$ for the product in $OH^n_C$ and $*_{C'}$ for the one in $OH^n_{C'}$, this means that $x *_C y = (-1)^{\eta_{C,C'}(|x|_B, |y|_B)} x *_{C'}y$ for all $x,y \in OH^n$. We compute
\begin{align*}
OF(C_{dba} \circ C_{dcb}\Id_{W(b)a}) &= (-1)^{\eta_{C,C'}(W(d)b,W(b)a) + \eta_{C,C'}(W(d)c,W(c)b)} OF(C'_{dba} \circ C'_{dcb}\Id_{W(b)a}), \\
OF(C_{dca} \circ \Id_{W(d)c}C_{cba}) &=  (-1)^{\eta_{C,C'}(W(d)c,W(c)a) + \eta_{C,C'}(W(c)b,W(b)a)} OF(C'_{dca} \circ \Id_{W(d)c}C'_{cba}),
\end{align*}
such that by definition of the associators, we get
\begin{equation}\label{eq:deta}
d\eta_{C,C'} = \phi_{C'}^{ch} - \phi_C^{ch},
\end{equation}
and thus, as $\phi_C^{ch} = \phi_{C'}^{ch}$ by hypothesis, $\eta_{C,C'}$ is a $2$-cocycle.
By Lemma~\ref{lem:geomrel}, we know that 
$$H^2(N(\mathcal G^n), \Z/2\Z) \simeq 0,$$
 and consequently, $\eta_{C,C'}$ is a coboundary. Hence, $\eta_{C,C'} = d\lambda_{C,C'}$ for a $\lambda_{C,C'} : N_1(\mathcal G^n) \rightarrow \Z/2\Z$. This means that the morphism
$$x \mapsto \lambda_{C,C'}(|x|_B) x  : OH^n_C \rightarrow  OH^n_{C'}, $$
is an isomorphism of quasialgebras.
\end{proof}

For the associative twisted algebra, the case is much simpler and all algebras are isomorphic. This means that the choice of $C$ and of twist $\tau_C$ does not matter in the end.

\begin{prop}
For all choices $C, C' \in \mathcal C^n$ and all choices of twists $\tau_C, \tau_{C'}$, there is an isomorphism
$\left(OH^n_C \otimes_\Z \Z[i]\right)_{\tau_C} \simeq\left(OH^n_{C'}\otimes_\Z \Z[i]\right)_{\tau_{C'}}.$
\end{prop}

\begin{proof}
For all $x, y \in OH^n$, we have
$$x *_{\tau_C} y = i ^{\tau_C(|x|,|y|)+2\eta_{C,C'}(|x|_B,|y|_B)-\tau_{C'}(|x|,|y|)} x *_{\tau_{C'}} y,$$
where $*_{\tau_C}$ is the product in $\left(OH^n_C \otimes_\Z \Z[i]\right)_{\tau_C}$, and we write
$$\theta_{C,C'} :=  \tau_C+2\eta_{C,C'}-\tau_{C'} : N_2(\mathcal G^n \times \mathcal Z) \rightarrow \Z/4\Z.$$
 We compute $d\theta_{C,C'} = \psi_C + 2d\eta_{C,C'} - \psi_{C'} \overset{(\ref{eq:deta})}{=} 0$, and by using similar arguments as in the proof of Proposition~\ref{prop:classodd}, we conclude that the two algebras are isomorphic. 
\end{proof}

\begin{cor}
The associative twisted algebra is uniquely determined up to isomorphism. We write it ${OH}^n_{\tau}$.
\end{cor}

\begin{rem}
Finding a twist is not an easy task, which can entail some serious difficulties for the construction of~${OH}^n_{\tau}$.
\end{rem}

\begin{prop}
The following three arc algebras are not isomorphic
$$H^n \otimes \Z[i] \not\simeq {OH}^n_{\tau} \not\simeq OH^n_C \otimes \Z[i].$$
\end{prop}

\begin{proof}
The first two are associative algebras as opposed to the last one. Let us begin with the case $n=2$. We know that the center of $H^2$ has graded rank $1+3q^2+2q^4$. Howewer, by a similar argument as in Proposition~\ref{prop:inccenter}, we have 
$$Z({OH}^2_{\tau}) \subset a(OH^2)a \oplus b(OH^2)b.$$
But ${OH}^2_{\tau}$ behaves like an exterior algebra on this subset, implying that elements anticommute and thus the graded rank of the center is $0$ in degree $2$. Therefore, $H^2$ and ${OH}^2_{\tau}$ have non-isomorphic centers and thus cannot be isomorphic as algebras.

\medskip

This can be extended to all $n \ge 2$. Suppose $x,y \in a({OH}^n_{\tau})a$ with $|x|_q=|y|_q = 1$, their products are given by
\begin{align*}
(x,y) &\mapsto \tau(|x|,|y|) x \wedge y,\\
(y,x) &\mapsto \tau(|y|,|x|) y \wedge x,
\end{align*}
since $OF(C_{aaa})$ is the product in the exterior algebra $\Ext^* V(W(a)a)$. However $|x|=|y|$, implying $\tau(|x|,|y|) = \tau(|y|,|x|)$ and  thus, $xy = -yx$.
\end{proof}

\begin{prop}
The odd center of ${OH}^2_{\tau}$ is not isomorphic to $OH(\mathfrak B_{2,2}, \Z[i])$.
\end{prop}

\begin{proof}
It is not hard to compute that the odd center of ${OH}^2_{\tau}$ is generated by the elements
$$OZ({OH}^2_{\tau}) = \left\langle \idaa + \idbb, \aoid - \idat, \boid - \idbt, \aoat, \bobt\right\rangle$$
and thus has graded rank $1 + 2q^2 + 2q^4$. However, we know that $OH(\mathfrak B_{2,2}, \Z[i])$ has graded rank $1 + 3q^2 + 2q^4$.
\end{proof}

Despite this result, it is easy to show that ${OH}^n_{\tau}$ contains a subalgebra which is isomorphic to $OH(\mathfrak B_{n,n}, \Z[i])$
by constructing an injective map, say $\tilde h$, similar to $h_0$ from Section~\ref{sec:ospringer}. 
  As $\bigoplus_a a(OH_\tau^n)a$ is isomorphic to an exterior algebra $\tilde h$ will be well defined.
    Moreover the different arc algebras being isomorphic modulo $2$, it is injective by the same
    reason as in Proposition~\ref{prop:injective}.

\section{Perspectives}

One natural application of the work in this paper could be the construction of odd Khovanov homology for tangles  
(Putyra-Shumako\-vitch's work in progress using the structure of quasialgebras~\cite{putyrashumakovitch}).  
The fact that the twist $\tau$ is not explicit may cause several technical difficulties in defining 
(and working with) $(OH^n_{\tau},OH^n_{\tau})$-bimodules. 

Another possibility consists of working with quasibimodules, that is bimodules with the associativity axiom given by an associator, as in~\cite{putyrapreprint}.
With such a theory at hand, it seems plausible that the braid group action on 
the category of complexes of $(OH^n_{C},OH^n_{C})$-quasibimodules  up to homotopy descends 
to an action of the $(-1)$-Hecke algebra from~\cite[Section 4]{laudarussell14} on its (odd) center, 
paralleling the even case (see~\cite[Section 5.3]{khovanov02}).  

\smallskip 

The fact that the twisted odd arc algebra $OH^n_{\tau}$ 
is defined over the Gaussian integers was for technical reasons. 
One question that we leave open is to find whether it is possible to twist $OH^n_C$ over the integers. 

\smallskip

The construction in this paper shares several features with Ehrig-Stroppel's Khovanov arc algebra of type $D$ 
from~\cite{ehrig-stroppel1}. 
It would be interesting to find a connection between these two arc algebras. 

\smallskip 

In~\cite{brundan-stroppel3}, an action of the 2-Kac-Moody of Rouquier~\cite{rouquier} 
(and therefore of Khovanov-Lauda's~\cite{khovanov-lauda1}) on Khovanov's arc algebras was constructed.   
The results of Rouquier on strong categorical actions~\cite{rouquier}, 
together with the fact that our associative arc algebra is 
not isomorphic to Khovanov's, imply that the 2-Kac-Moody algebra does not act on it. 
It seems plausible to expect that the odd arc algebra admits an action of an algebra akin to  
Brundan and Kleshchev's Hecke-Clifford superalgebra 
from~\cite{brundan-kleshchev}, which could be seen as a super counterpart of the 
cyclotomic KLR algebra.     

\smallskip 

Another challenging problem we would like to mention is to find the representation-theoretic 
context (category $\mathcal{O}$) for the odd arc algebras. 
The analogy with~\cite{ehrig-stroppel1} and~\cite{ehrig-stroppel2}, taken together with~\cite{ellis-lauda} 
and the results in~\cite[Section 5.1]{ehrig-stroppel-tubbenhauer2} might 
suggest a connection to category $\mathcal{O}$ for Lie superalgebras.

\bibliographystyle{habbrv}
\bibliography{NaisseVaz_HnOdd}

\end{document}